\documentclass[11pt,reqno]{amsart}
\usepackage[utf8]{inputenc}
\usepackage[british]{babel}
\usepackage{amsmath,amssymb,amsthm,bbm,mathtools,calc,verbatim,enumerate,tikz,url,mathrsfs,fullpage,hyperref,bm,scalerel}
\usepackage[numbers,longnamesfirst,sort]{natbib}

\usepackage{lmodern}
\usepackage[T1]{fontenc}
\usepackage[activate={true,nocompatibility},final,kerning=true]{microtype}

\renewcommand{\le}{\leqslant}
\renewcommand{\leq}{\leqslant}
\renewcommand{\ge}{\geqslant}
\renewcommand{\geq}{\geqslant}
\newtheorem{theorem}{Theorem}[section]
\newtheorem{proposition}[theorem]{Proposition}
\newtheorem{remark}[theorem]{Remark}
\newtheorem{definition}[theorem]{Definition}
\newtheorem{lemma}[theorem]{Lemma}

\newtheorem{conjecture}[theorem]{Conjecture}
\newtheorem{question}[theorem]{Question}
\newtheorem*{claim*}{Claim}

\renewcommand{\P}[1]{\operatorname{Perm}(#1)}
\newcommand{\eps}{\varepsilon}
\newcommand{\Bin}{\operatorname{Bin}}
\newcommand{\dist}{\operatorname{dist}}
\newcommand{\iso}{\mathrm{i}}
\newcommand{\tran}{\intercal}
\newcommand{\cart}[1]{\vcenter{#1\kern.2ex\hbox{$\square$}\kern.2ex}}
\newcommand{\bigcart}{\mathop{\mathchoice{\cart\Large}{\cart\large}{\cart\normalsize}{\cart\small}}}

\title{The evolution of random subgraphs of the permutahedron}

\author{Maur\'icio Collares}
\address{Instituto de Matem\'atica, Estat\'istica e Ci\^encia da Computa\c{c}\~ao, Universidade de S\~ao Paulo, Rua do Mat\~ao 1010, 05508-090 S\~ao Paulo, Brazil}
\email{collares@ime.usp.br}

\author{Joseph Doolittle}
\address{RobotDreams GmbH, Schubertstra{\ss}e 6a, 8010 Graz, Austria}
\email{nerdyjoe@gmail.com}

\author{Joshua Erde}
\address{School of Mathematics, University of Birmingham, Birmingham, B15 2TT, UK}
\email{j.erde@bham.ac.uk}

\date{}

\begin{document}
\begin{abstract}
  In their seminal paper introducing the theory of random graphs, Erd\H{o}s and R\'{e}nyi considered the evolution of the structure of a random subgraph of $K_n$ as the density increases from $0$ to $1$, identifying two key points in this evolution --- the percolation threshold, where the order of the largest component seemingly jumps from logarithmic to linear in size, and the connectivity threshold, where the subgraph becomes connected. Similar phenomena have been observed in many other random graph models, and in particular, works of Ajtai, Koml\'{o}s and Szemer\'{e}di and of Erd\H{o}s and Spencer determine corresponding thresholds for random subgraphs of the hypercube.

  The permutahedron, like the hypercube, has many equivalent representations, and arises as a natural object of study in many areas of combinatorics. In particular, as a highly-symmetric simple polytope, like the $n$-simplex and $n$-cube, percolation on the permutahedron arises naturally as an analogue to the Erd\H{o}s--R\'{e}nyi random graph and the percolated hypercube.

  We determine the percolation threshold and the connectivity threshold for random subgraphs of the permutahedron. Along the way we develop a novel graph exploration technique which can be used to find exponentially large clusters after percolation in high-dimensional geometric graphs, and we initiate the study of the isoperimetric properties of the permutahedron.
\end{abstract}
\maketitle

\section{Introduction}
Given a graph $G$ and a probability $p \in (0,1)$ the \emph{percolated random subgraph} $G_p$ is the random subgraph of $G$ obtained by including each edge of $G$ independently with probability $p$.
In a sense, $G_p$ represents the \emph{typical} structure of a subgraph of $G$ with relative density $p$.
This model was first considered in the context of statistical physics, where it was used to model the flow of liquid through a porous medium whose channels could be randomly blocked~\cite{BH57}.
There is a natural way to couple the random subgraphs $G_p$ for $p \in [0,1]$ so that they are increasing (as subgraphs) as $p$ increases, and in this way we can think of the structure of $G_p$ as \emph{evolving} as we increase $p$ from $0$ to $1$, with $G_p$ gradually growing from an empty graph to the full \emph{host} graph $G$.

Perhaps the simplest percolation model is when we take the host graph $G$ to be a complete graph $K_n$, in which case we recover the classical \emph{binomial random graph model} $G(n,p)$.
The evolution of the structure of $G(n,p)$ was the topic of one of the earliest papers on random graphs, in which Erd\H{o}s and R\'{e}nyi~\cite{ER60} showed that the component structure of $G(n,p)$\footnote{In fact, their result was stated in terms of the \emph{uniform} random graph model $G(n,m)$.}
undergoes a dramatic phase transition when $p$ is around $\frac{1}{n}$.
More precisely, given a constant $c>1$, let us define $\gamma(c)$ to be the unique solution in $(0,1)$ of the equation
\begin{equation}\label{e:survival-prob}
  \gamma(c)=1-\exp\bigl(-c\gamma(c)\bigr).
\end{equation}
\begin{theorem}[\cite{ER60}]\label{t:ER}
  Let $c > 0$ be a constant and let $p =\frac{c}{n}$.
  Then, with high probability\footnote{With probability tending to one as $n \to \infty$.} (whp)
  \begin{enumerate}[$(i)$]
  \item\label{i:ERsub} if $c<1$, the largest component of $G(n,p)$ has order\footnote{Throughout the paper, unless otherwise specified, all logarithms are taken with respect to the natural base $e$.} $\Theta(\log n)$; and,
  \item\label{i:ERsup} if $c>1$, there exists a \emph{giant} component in $G(n,p)$ of order $(1+o(1))\gamma(c)n$, where $\gamma$ is defined according to~\eqref{e:survival-prob}, and the second-largest component has order $\Theta(\log n)$.
  \end{enumerate}
\end{theorem}
We note that Theorem~\ref{t:ER} is only a very broad description of the phase transition, and much more is known about the structure of $G(n,p)$ when $p$ is close to the critical point (see, for example,~\cite{B84,L90,LPW94}).

It turns out that $G(n,p)$ is in some sense partially \emph{universal} in this class of percolation models, in the sense that for many $n$-regular\footnote{Whilst the complete graph $K_n$ is actually $(n-1)$-regular, we are viewing the phase transition on a coarse enough level that the difference is negligible.} host graphs $G$, the component structure of the percolated subgraph $G_p$ undergoes a quantitatively similar phase transition as that described in Theorem~\ref{t:ER} when $p$ is around $\frac{1}{n}$.
In these models, the broad scale structure of $G_p$, under the right scaling, seems to be independent of the underlying geometry of the host graph $G$.
A particularly notable example of this phenomenon was shown to occur in the $n$-dimensional hypercube $Q^n$ by Ajtai, Koml\'{o}s and Szemer\'{e}di~\cite{AKS81}, answering a well-known conjecture of Erd\H{o}s and Spencer~\cite{ES79}.

\begin{theorem}[\cite{AKS81, BKL92}]\label{t:cube}
  Let $c>0$ be a constant and let $p=\frac{c}{n}$.
  Then, whp,
  \begin{enumerate}[$(i)$]
  \item if $c<1$, the largest component of $Q^n_p$ has order $\Theta(n)$; and,
  \item if $c>1$, there exists a giant component in $Q^n_p$ of order $(1+o(1))\gamma(c)2^n$, where $\gamma$ is defined according to~\eqref{e:survival-prob}, and the second-largest component has order $\Theta(n)$.
  \end{enumerate}
\end{theorem}
Note here, in comparison to Theorem~\ref{t:ER}, that $\log |Q^n|= \Theta(n)$.
Hence, we see a quantitatively similar phase transition in the component structure.
When $c <1$, which we call the \emph{subcritical regime}, whp all the components of the percolated subgraph are small, of logarithmic order, whereas when $c>1$, which we call the \emph{supercritical regime}, whp a giant component of linear order emerges, covering the same asymptotic proportion of the vertices as in Theorem~\ref{t:ER}~\eqref{i:ERsup}, and all other components are again of at most logarithmic order.

It can be seen that $\gamma \colon (1,\infty) \to \mathbb{R}$ is a continuous, increasing function and $\gamma(c) \to 1$ as $c\to \infty$, and so as $c$ increases, less and less of the percolated subgraph lies outside the giant component.
Hence, already when $p=\frac{\omega(1)}{n}$ the giant component covers all but a vanishing proportion of the vertices.
However, another classic result of Erd\H{o}s and R\'enyi~\cite{ER60} shows that we have to wait significantly longer until $G(n,p)$ becomes connected.

Indeed, a connected graph cannot have any isolated vertices and hence, since $K_n$ is an $(n-1)$-regular, $n$-vertex graph, we should not expect the graph to be typically connected until the expected number of isolated vertices $n(1-p)^{n-1}$ becomes small.
\begin{theorem}[\cite{ER60}]\label{t:ERconn}
  Let $p=p(n)$ and let $\lambda(n,p) =n(1-p)^{n-1}$.
  Then
  \[
    \lim_{n \to \infty}\mathbb {P}(G(n,p) \text{ is connected} ) = \begin{cases}
      0 &\text{ if }  \lambda \to  \infty;\\
      e^{-c} &\text{ if } \lambda  \to c\geq 0.\\
    \end{cases}
  \]
\end{theorem}
The corresponding result in $Q^d_p$ is originally due to Burtin~\cite{B77}, where the above heuristic suggests the threshold should occur when the expected number of isolated vertices, which can be seen to be $2^n(1-p)^n$, becomes small.
\begin{theorem}[\cite{ES79}]\label{t:Cubeconn}
  Let $p=p(n)$ and let $\lambda(n,p) = 2^n(1-p)^n$.
  Then
  \[
    \lim_{n \to \infty}\mathbb {P}(Q^d_p \text{ is connected} ) = \begin{cases}
      0 &\text{ if }  \lambda(n,p) \to \infty;\\
      e^{-c} &\text{ if }  \lambda(n,p) \to c \geq 0.\\
    \end{cases}
  \]
\end{theorem}

In this paper we study the evolution of the structure of a random subgraph of a different graph, the $n$-dimensional \emph{permutahedron}.
The permutahedron is a well-studied combinatorial object which has multiple equivalent representations, perhaps the simplest of which is its representation as a (left-multiplication) Cayley graph of the symmetric group\footnote{We adopt the convention that, for $\pi, \sigma \in S_{n+1}$, $\sigma\pi$ is the permutation given by $(\sigma\pi)(i) = \sigma(\pi(i))$.} $S_{n+1}$, generated by the adjacent transpositions.
\begin{definition}\label{def:permutahedron}
  The \emph{$n$-dimensional permutahedron}, denoted by $\P{n}$, is the graph on vertex set $S_{n+1}$ with
  \[
    E(\P{n})=\bigl\{\,\{\pi, \tau_i\pi\} : \pi \in S_{n+1}\,,\, 1 \leq i \leq n\,\bigr\},
  \]
  where $\tau_i$ is the transposition $(i, i+1)$.
\end{definition}
In other words, two permutations are connected by an edge if they differ by swapping two values that differ by $1$.
We note that the generating set of the Cayley graph above are precisely the generators in the presentation of $S_{n+1}$ as a \emph{Coxeter group} (for background on Coxeter groups, see e.g.,~\cite{D08}).

However, $\P{n}$ has a number of other equivalent definitions.
It is the $1$-skeleton of a polytope $P(n)$, which can be described as the convex hull of the points in $\mathbb{R}^{n+1}$ whose coordinates are given by the word representations of all permutations in $S_{n+1}$, but which can also be seen to be a \emph{zonotope}, a Minkowski sum of line segments.
Note that, since $P(n)$ lives in an affine hyperplane of $\mathbb{R}^{n+1}$, it has intrinsic dimension $n$.
See~\cite{Z95} for more on zonotopes and convex polytopes in general.
The permutahedron can also be defined as the covering graph of a particular lattice, the \emph{weak order} or \emph{weak Bruhat lattice}, on $S_{n+1}$.

A weak analogy can be drawn here with the $n$-dimensional hypercube, which also has numerous equivalent representations, as the standard Cayley graph of the group $\mathbb{Z}_2^n$, with generators which represent the group as a Coxeter group, as the $1$-skeleton of the hypercube polytope, which is also a zonotope, or as the covering graph of the lattice of subsets of $[n]$.
As with the hypercube, due to its many equivalent representations, the permutahedron arises naturally in a variety of combinatorial contexts.
However, whilst the graph theoretical properties of the hypercube have been well-studied, from both the probabilistic and the extremal viewpoint~\cite{HHW88}, the permutahedron has mostly been considered in the context of algebraic and enumerative combinatorics, and much less is known about the structure of the permutahedron as a graph.

One reason that it is perhaps natural to study bond percolation on the permutahedron, is in the context of percolation on polytopes.
Indeed, both $K_n$ and $Q^n$ are $1$-skeletons of particularly symmetric polytopes, the $n$-simplex and the $n$-cube, respectively, and in some sense these are perhaps the most natural polytopes on which to study this percolation model since they are \emph{simple}\footnote{That is, a $d$-dimensional polytope in which each vertex is in exactly $d$ edges.}, and so the polytope is (combinatorially) determined by its $1$-skeleton~\cite{BML87,K88}, and \emph{regular}\footnote{That is, a polytope whose symmetry group acts transitively on its flags.}, and hence as symmetric as possible.
Apart from some sporadic examples in low dimensions, there is only one other infinite family of regular polytopes, the \emph{cross-polytope}, which is not simple, and whose $1$-skeleton is the \emph{cocktail party graph}, which is a complete graph with a perfect matching removed.
However, since this graph is so close to a clique, similar methods as in~\cite{ER60} can be used to show analogues of Theorems~\ref{t:ER} and~\ref{t:ERconn}.

Whilst the permutahedra are not regular polytopes, they are a family of \emph{uniform}\footnote{That is, a polytope whose symmetry group acts transitively on its vertices.} simple polytopes in each dimension, and so another natural family of highly symmetric polytopes whose combinatorial structure is determined by their $1$-skeletons.

Our first main result is that the percolated permutahedron undergoes a quantitatively similar phase transition around the point $\frac{1}{n}$ as the binomial random graph $G(n,p)$.
\begin{theorem}\label{t:percolationthreshold}
  Let $c>0$ be a constant and let $p=\frac{c}{n}$.
  Then, whp,
  \begin{enumerate}[$(i)$]
  \item\label{i:subcritical} if $c <1$, the largest component of $\P{n}_p$ has order $\Theta(n\log n)$; and
  \item\label{i:supercritical} if $c > 1$, there exists a giant component in $\P{n}_p$ of order $(1+o(1))\gamma(c) (n+1)!$, where $\gamma$ is defined according to~\eqref{e:survival-prob}, and the second-largest component has order $\Theta(n\log n)$.
  \end{enumerate}
\end{theorem}

Note that, $\log |\!\P{n}| = \Theta(n \log n)$, and so Theorem~\ref{t:percolationthreshold} is again quantitatively similar to Theorems~\ref{t:ER} and~\ref{t:cube}.
In the language of~\cite{DEKK22} we say that the permutahedron exhibits the \emph{Erd\H{o}s--R\'{e}nyi component phenomenon}.
Here, the fact that the order of $\P{n}$ is superexponential in its regularity causes significant difficulties compared to the case of the complete graph or hypercube.
The proof of Theorem~\ref{t:percolationthreshold} utilises a novel exploration process, which we call \emph{projection-first search} (see Section~\ref{s:PFS}), which is applicable to many high-dimensional geometric graphs, and significantly strengthens and simplifies the analysis of the distribution of `large' clusters after percolation in these graphs.
Furthermore, this process can be used to effectively enumerate small subgraphs in such graphs, which is useful for determining the existence of logarithmic sized components in both regimes.

We also consider the connectivity threshold in this model.
Note that, since $\P{n}$ is an $n$-regular, $(n+1)!$-vertex graph, the expected number of isolated vertices in $\P{n}_p$ is $(n+1)!\,(1-p)^n$.
\begin{theorem}\label{t:connectivitythreshold}
  Let $p=p(n)$ and let $\lambda(n,p) = (n+1)!\,(1-p)^n$.
  Then
  \[
    \lim_{n \to \infty}\mathbb {P}(\P{n}_p \text{ is connected} ) = \begin{cases}
      0 &\text{ if }  \lambda(n,p) \to \infty; \\
      e^{-c} &\text{ if }  \lambda(n,p) \to c \geq 0.\\
    \end{cases}
  \]
\end{theorem}
Both the superexponential order of $|\!\P{n}|$ and the lack of explicit control of the number of subgraphs and their expansion properties make it hard to argue as in the proof of Theorem~\ref{t:Cubeconn}.
Here, instead, we give a novel proof of the connectivity threshold which uses information about the component structure and the distribution of vertices in the giant component derived in the proof of Theorem~\ref{t:percolationthreshold}, which again should have applications to similar questions in other percolation models.
Using standard methods, it is easy to give a corresponding hitting time result.
We note that a similar approach to the connectivity threshold in $Q^n$ and other high-dimensional product graphs was recently developed by Diskin and Krivelevich~\cite{DK24} and Diskin and Geisler~\cite{DG24}.

As a tool to prove Theorems~\ref{t:percolationthreshold} and~\ref{t:connectivitythreshold}, but also as an interesting problem in its own right, we consider the \emph{isoperimetric properties} of $\P{n}$.
In many contexts, the isoperimetric properties of the host graph $G$ have been key to understanding the component structure of the percolated subgraph $G_p$~\cite{AKS81,BKL92,DEKK22,FKM04}, and in turn the isoperimetric properties of the percolated subgraph $G_p$ have been key to understanding its internal structure~\cite{BKW14, DEKK23+,DLP14,EKK22}.

Given a graph $G$, we define
\[
  i_k(G) := \min_{S\subseteq V(G), |S|=k}\frac{|\partial(S)|}{|S|},
\]
where $\partial(S)$ is the edge boundary of $S$, and the \emph{edge-isoperimetric constant}, sometimes known as the \emph{Cheeger constant}, given by
\[
  i(G) := \min_{1 \le k\le |V(G)|/2} i_k(G).
\]
For a general graph, determining $i(G)$ or even $i_k(G)$ is known to be a computationally hard problem~\cite{GJS76} but the exact value, or asymptotics, have been determined for many families of lattice-like graphs~\cite{BI91,AB95,L15,BE18}.

The isoperimetric problem does not seem to have been well-studied in $\P{n}$, despite it being mentioned as an open problem in the monograph of Harper~\cite{H04}.
Nevertheless, due to the many explicit representations of the permutahedron and its high level of symmetry, results in the literature can be used to give an essentially optimal bound for the expansion of small sets, as well as a general bound on the isoperimetric constant of $\P{n}$.

\begin{proposition}\label{p:edge-iso-general}
  The edge-isoperimetric constant of $\P{n}$ satisfies
  \[
    \iso\bigl(\P{n}\bigr) = \Omega\Bigl(\frac{1}{n^2}\Bigr).
  \]
  Moreover, for every $k\in [2^n]$,
  \[
    \iso_k\bigl(\P{n}\bigr) \geq n - \log_2 k.
  \]
\end{proposition}

The first bound here is a straightforward consequence of known bounds on the spectrum of the permutahedron~\cite{B94} combined with (a combinatorial form of) Cheeger's inequality~\cite{AM85}, which relates the expansion of a graph to its spectrum.
The second bound mirrors the classic edge-isoperimetric inequality for the hypercube due to Harper~\cite{H64}, and follows from that fact that $\P{n}$ is a \emph{partial cube}, an isometric subgraph of a hypercube, see Lemma~\ref{l:zonotopepartialcube}.

The paper is structured as follows.
In Section~\ref{s:prelim} we introduce some notation and give some preliminary lemmas.
In Section~\ref{s:iso} we discuss the isoperimetric properties of $\P{n}$ and give a proof of Proposition~\ref{p:edge-iso-general}.
In Section~\ref{s:PFS} we introduce the projection-first search algorithm, which uses projection lemmas proved in Section~\ref{s:projection}, and give some consequences for percolation on $\P{n}$ as well as for small subgraph counts.
In Sections~\ref{s:perc} and~\ref{s:conn} we discuss the percolation and connectivity thresholds and prove Theorems~\ref{t:percolationthreshold} and~\ref{t:connectivitythreshold}, respectively.
Finally, in Section~\ref{s:discuss} we discuss avenues for future research.

\section{Preliminaries}\label{s:prelim}
In this section, we collect some tools and definitions which will be useful later.
The graph theory terminology we use is standard (see, e.g.,~\cite{B09}).
For polytope theory, we follow the terminology in~\cite{Z95}, but provide a summary of the key definitions.
We also include some key properties of the permutahedron, as well as lemmas about random graphs, trees and branching processes.

\subsection{A brief introduction to polytopes}
A $\mathcal{V}\emph{-polytope}$ is the convex hull of a finite set of points of $\mathbb{R}^n$ for some $n \in \mathbb{N}$.
An $\mathcal{H}\emph{-polytope}$ in $\mathbb{R}^n$ is a bounded set of the form $\{\bm{x} \in \mathbb{R}^n : \bm{A}\bm{x} \geq \bm{b}\}$ for some $\bm{A} \in \mathbb{R}^{k\times n}$ and $\bm{b} \in \mathbb{R}^k$, that is, a bounded intersection of halfspaces.
It can be shown that every $P \subseteq \mathbb{R}^n$ which is either an $\mathcal{H}$-polytope or a $\mathcal{V}$-polytope is both.
Such sets are called \emph{polytopes}.
The \emph{dimension} of a non-empty polytope $P$, denoted by $\dim(P)$, is the dimension of the smallest affine plane in which it is contained, and it is useful to define $\dim(\emptyset) = -1$.

Let $P \subseteq \mathbb{R}^n$ be a polytope.
If $\bm{a} \in \mathbb{R}^n$ and $b \in \mathbb{R}$ are such that $\bm{a}^\tran\bm{x} \geq b$ holds for every $\bm{x} \in P$, then the set $F = P \cap \{\,\bm{x} \in \mathbb{R}^n : \bm{a}^\tran\bm{x} = b\,\}$ is called a \emph{face} of $P$, and the inequality $\bm{a}^\tran \bm{x} \geq b$ \emph{defines} $F$.
One may easily check that $\emptyset$ and $P$ are always faces of $P$.
Other faces, if any, are called \emph{proper}.
Clearly, a face of a polytope is a polytope itself.
A face of $P$ of dimension $0$, $1$ or $\dim(P) - 1$ is called a \emph{vertex}, an \emph{edge} or a \emph{facet} of $P$, respectively.
The set of vertices of $P$ is denoted by $V(P)$.
Any proper face of $P$ may be written as an intersection of facets of $P$, and a face of a face of $P$ is also a face of $P$.
The faces of $P$ form a partially ordered set under inclusion, called the \emph{face lattice} of $P$.
Two polytopes are said to be \emph{combinatorially equivalent} if their face lattices are isomorphic.

Let $\bm{v}$ be a vertex of $P$.
The \emph{vertex figure} of $P$ at $\bm{v}$, denoted by $P / \bm{v}$, is the polytope obtained by intersecting $P$ with a hyperplane that separates $\bm{v}$ from the other vertices of $P$.
Formally, take an arbitrary inequality $\bm{a}^\tran \bm{x} \geq b$ defining $\{\bm{v}\}$, and any $b' > b$ such that $\bm{a} \bm{v'} \geq b'$ holds for every vertex $\bm{v'} \neq \bm{v}$ of $P$.
We define
\[
  P / \bm{v} := P \cap \{\,\bm{x} \in \mathbb{R}^n : \bm{a}^\tran \bm{x} = b'\,\}.
\]
Since $P / \bm{v}$ is the intersection of $P$ with a hyperplane, it is also a polytope.
Denoting by $\operatorname{aff}(S)$ the affine span of a set $S$, we have that, for every $0 \leq k \leq \dim(P)$, the mapping
\[ F \mapsto P \cap \operatorname{aff}\bigl(\{\bm{v}\} \cup F\bigr) \]
is a bijection between $(k-1)$-dimensional faces of $P/\bm{v}$ and $k$-dimensional faces of $P$ containing $\bm{v}$ (\cite[Proposition 2.4]{Z95}).
In particular, this implies that $P / \bm{v}$ is well-defined up to combinatorial equivalence, despite the arbitrary choices of $\bm{a}$ and $b'$.

Given $\bm{x}, \bm{y} \in \mathbb{R}^n$, we denote the line segment joining them by $[\bm{x}, \bm{y}]$.
A graph, called the \emph{$1$-skeleton} of $P$, may be obtained by connecting distinct vertices $\bm{v}$, $\bm{w}$ of $P$ whenever $[\bm{v}, \bm{w}]$ is a face of $P$.
A $d$-dimensional polytope $P$ is \emph{simple} if its $1$-skeleton is a $d$-regular graph.
This is not to be confused with the definition of a regular polytope, which will not be used in the rest of this paper.
It can be shown that $P$ is simple if and only if $P / \bm{v}$ is a simplex (that is, its vertices are affinely independent, or equivalently every subset of vertices corresponds to a different face) for every vertex $\bm{v}$ of $P$.
By a celebrated theorem of Blind and Mani-Levitska~\cite{BML87} (see also~\cite{K88}), simple polytopes are determined by their $1$-skeletons (up to combinatorial equivalence).

\subsection{Zonotopes}\label{s:zonotope}
A \emph{zonotope} $Z$ is the image of a cube $[-1, 1]^p$ under an affine projection.
Writing the projection as $\pi(\bm{x}) = \bm{V}\bm{x} + \bm{z}$ and denoting the columns of $\bm{V}$ by $\bm{v_1}, \ldots, \bm{v_p}$, we see that a set $Z \subseteq \mathbb{R}^n$ is a zonotope if and only if there exist $p \geq 0$ and vectors $\bm{z}, \bm{v_1}, \ldots, \bm{v_p}$ such that $Z = \{\bm{z}\} + [-\bm{v_1}, \bm{v_1}] + \cdots + [-\bm{v_p}, \bm{v_p}]$, where $+$ denotes the Minkowski sum.
The vectors $\bm{v_1}, \ldots, \bm{v_p}$ are called \emph{generators} of $Z$.
For any zonotope, one may choose generators such that all $\bm{v_i}$ are nonzero and no two $\bm{v_i}$ are parallel (as the Minkowski sum is commutative and the sum of two parallel segments is a segment).
Every zonotope is a polytope, and a \emph{simple} zonotope is a zonotope which is a simple polytope.

The following result can be deduced in a relatively straightforward manner from known results in the literature, but we were not able to find a direct statement of this fact.
A \emph{partial cube} is a graph that can be isometrically embedded into a hypercube.

\begin{lemma}\label{l:zonotopepartialcube}
  Let $Z$ be a zonotope.
  Then the $1$-skeleton of $Z$ is a partial cube.
\end{lemma}
\begin{proof}
  The statement is trivial if $Z$ has dimension $0$.
  Without loss of generality, we may translate $Z \subseteq \mathbb{R}^n$ so that it is the image of $Q = [-1, 1]^k$ under a linear projection $\pi$.
  We will use $\bm{V}$ to denote the matrix of $\pi$, whose columns we will denote by $\bm{v_1}, \ldots, \bm{v_k}$.
  As mentioned, we may assume no $\bm{v_i}$ is zero and no two $\bm{v_i}$ are parallel.
  We wish to show that $\pi^{-1}$ restricted to $V(Z)$ is a function between vertices of the $1$-skeletons of $Z$ and $Q$, and a graph isometry to its image.

  Let $F$ be a face of $Z$ and $\bm{a}^\tran \bm{z} \geq b$ an inequality that defines it (with $\bm{a} \in \mathbb{R}^n$, $b \in \mathbb{R}$).
  The corresponding inequality $(\bm{a}^\tran \bm{V}) \bm{x} \geq b$ shows that $\pi^{-1}(F)$ is a face of $Q$.
  Using the well-known fact that every proper face of a cube is a subcube, we can say a little more about preimages of low-dimensional faces of $Z$.

  Indeed, since we chose $\pi$ such that no $\bm{v_i}$ is zero, the image under $\pi$ of every subcube of $Q$ of dimension at least $1$ has affine dimension at least $1$ and hence $\pi^{-1}(\bm{u})$ is a vertex of $Q$ for every vertex $\bm{u}$ of $Z$.
  Similarly, since no $\bm{v_i}$ is zero and no two $\bm{v_i}$ are parallel, the image under $\pi$ of every subcube of $Q$ of dimension at least $2$ has affine dimension at least $2$ and hence $\pi^{-1}(e)$ is an edge of $Q$ for every edge $e$ of $Z$.
  This implies that $\dist(\pi^{-1}(\bm{u}), \pi^{-1}(\bm{w})) \leq \dist(\bm{u}, \bm{w})$, since we may `pull back' shortest paths in $Z$ to the hypercube.
  It remains to show the opposite inequality.

  Note that, by definition, every $\bm{u} \in V(Z)$ admits an $\bm{a_u} \in \mathbb{R}^n$ such that $\bm{u}$ is the unique minimiser of $\bm{a_u}^\tran \bm{z}$ among all $\bm{z} \in Z$.
  If the $i$th coordinate of $\pi^{-1}(\bm{u})$ is $1$ (resp.\ $-1$), it is easy to see that $\bm{a_u}^\tran \bm{v_i}$ is negative (resp.\ positive).
  For $\bm{u}, \bm{w} \in V(Z)$, let $I(\bm{u}, \bm{w}) = \{i \in [k] : (\pi^{-1}(\bm{u}))_i \neq (\pi^{-1}(\bm{w}))_i\}$, and observe that $\bm{a_u}^\tran\bm{v_i}$ and $\bm{a_w}^\tran \bm{v_i}$ have the same sign if and only if $i \not\in I(\bm{u},\bm{w})$.
  Since $\dist(\pi^{-1}(\bm{u}), \pi^{-1}(\bm{w})) = |I(\bm{u}, \bm{w})|$,
  our remaining goal is to show the existence of a path connecting $\bm{u}, \bm{w} \in V(Z)$ of length $|I(\bm{u}, \bm{w})|$.

  We will do so by induction on $m = |I(\bm{u}, \bm{w})|$, the case $m = 0$ being trivial.
  Given distinct $\bm{u}, \bm{w} \in V(Z)$, let $\bm{\hat{a}_t} = (1-t)\bm{a_u} + t\bm{a_w}$ for all $t \in [0,1]$, and take the smallest $\eps > 0$ such that the set $J_\eps = \{j \in I(\bm{u}, \bm{w}) : \bm{\hat{a}_\eps}^\tran \bm{v_j} = 0\}$ is nonempty.
  Let $F$ be the face minimising $\bm{\hat{a}_\eps}^\tran\bm{z}$ among all $\bm{z} \in Z$, and observe that the face $\pi^{-1}(F)$ minimises $f(\bm{x}) = \bm{\hat{a}_\eps}^\tran\bm{V}\bm{x}$ over all $\bm{x} \in Q$.
  For all $j\not\in J_\eps$, $\bm{a_u}^\tran\bm{v_j}$ and $\bm{\hat{a}_\eps}^\tran\bm{v_j}$ have the same sign, and therefore $\bm{u} \in F$ and $I(\bm{u}, \bm{z}) \subseteq J_\eps$ for every $\bm{z} \in V(F)$.
  Moreover, for every $j \in J_\eps$, $\pi^{-1}(F)$ contains the edge in the $j$th direction containing $\pi^{-1}(\bm{u})$, since $f$ is constant on this edge.
  Therefore, $\dim(\pi^{-1}(F)) \geq 1$, and thus $\dim(F) \geq 1$ as argued before.

  Taking an arbitrary edge $[\bm{u}, \bm{z}] \subseteq F$, we have $I(\bm{u}, \bm{z}) \subseteq J_\eps$ and $|I(\bm{u}, \bm{z})| = 1$, and thus $|I(\bm{z}, \bm{w})| = m-1$.
  By induction, there is a path between $\bm{z}$ and $\bm{w}$ of length $m-1$.
  Extending it with the edge $\bm{u}\bm{z}$, we obtain the desired path, finishing the proof.
\end{proof}

\subsection{Coxeter groups}
A group $\Gamma$ which has a presentation of the form
\[
  \bigl\langle r_1, r_2, \ldots , r_n \, \bigm| \, (r_i r_j)^{m_{i,j}} = e \bigr\rangle,
\]
where $m_{i,i} = 1$ for all $i$ and $m_{i,j} = m_{j,i} \geq 2$ if $i \neq j$, is said to be a \emph{Coxeter group}.
Given such a presentation we call $(\Gamma,\{r_1,\ldots, r_n\})$ a \emph{Coxeter system}.

Given a Coxeter system $(\Gamma,\{r_1,\ldots, r_n\})$ and a subset $I \subseteq [n]$ let us write $\Gamma_I$ for the subgroup of $\Gamma$ generated by $\{r_i : i \not\in I \}$.
We will need the following elementary fact about Coxeter groups.

\begin{lemma}[{\cite[Theorem 4.1.6]{D08}}]\label{l:Coxeter}
  Let $(\Gamma, \{r_1,\ldots, r_n\})$ be a Coxeter system and let $I,J \subseteq [n]$.
  Then
  \[
    \Gamma_I \cap \Gamma_J = \Gamma_{I \cup J}.
  \]
\end{lemma}

Recall from elementary group theory that, for any subgroup $H$ of $\Gamma$ and any $a \in \Gamma$, the coset of $H$ containing $a$ can be written as $\{b \in \Gamma: a^{-1}b \in H\}$.
This implies that, for two subgroups $H$, $K$ of $\Gamma$, the intersection of a coset of $H$ and a coset of $K$ is either empty or a coset of the intersection subgroup $H \cap K$.
Hence, by Lemma~\ref{l:Coxeter}, the intersection of a coset of $\Gamma_I$ and a coset of $\Gamma_J$ is either empty or a coset of $\Gamma_{I\cup J}$.
In particular, for any tuple $(g_1,g_2,\ldots, g_n)$ there is at most one element $h \in \bigcap_{i=1}^n g_i \Gamma_{\!\{i\}}$, and so we may identify $\Gamma$ with a particular subset of $\Gamma^{n}$ in this fashion, giving a coordinate representation of sorts, which will be useful later in the paper.

\subsection{Properties of the permutahedron}
We will need to use some basic facts about the structure of the permutahedron and about its various representations.
\subsubsection{Representations of the permutahedron}
In the paper it will sometimes be convenient to view the permutahedron as a polytope, which we will denote by $P(n)$, in order to argue geometrically.
In this paper, we will view $P(n)$ in three different ways.

The most useful to us is as a $\mathcal{V}$-polytope, given by the convex hull of the permutation vectors \[\bigl\{\,\bigl(\pi(1),\pi(2),\ldots, \pi(n+1)\bigr) : \pi \in S_{n+1}\,\bigr\}\] in $n+1$ dimensions.
Note that it is apparent that the polytope $P(n)$ is contained in an $n$-dimensional subspace of $\mathbb{R}^{n+1}$, and so $P(n)$ is an $n$-dimensional polytope.
To understand $P(n)$'s faces, it is useful to view it as an $\mathcal{H}$-polytope given by the set of points $\bm{x}$ such that
\begin{equation}\label{e:Hpolytope}
  \bm{1}^\tran\bm{x} = \binom{n+2}{2} \quad\text{and}\quad \bm{1}_I^\tran\, \bm{x} \geq \binom{|I|+1}{2}\text{ for every }\emptyset \subsetneq I \subsetneq [n+1],
\end{equation}
where $\bm{1}_I \in \mathbb{R}^{n+1}$ is the characteristic vector of the set $I$.

Another useful presentation of the polytope $P(n)$ is as zonotope, arising as the Minkowski sum
\[
  \frac{n+2}{2} \cdot \bm{1} + \sum_{1 \leq j < i \leq n+1} [-\bm{v_{ij}}, \bm{v_{ij}}],
\]
where $\{\bm{e}_1, \ldots, \bm{e}_{n+1}\}$ is the canonical basis of $\mathbb{R}^{n+1}$, $\bm{1} = (1, \ldots, 1) \in \mathbb{R}^{n+1}$, and $\bm{v_{ij}} = (\bm{e_i} - \bm{e_j})/2$.

Finally, it will also be important to us to consider the presentation of the graph $\P{n}$ as a Cayley graph arising from a \emph{Coxeter system}.
It is apparent from Definition~\ref{def:permutahedron} that $\P{n}$ is the (left-multiplication) Cayley graph of the symmetric group $S_{n+1}$ whose generating set $S = \{\tau_i : i \in [n]\}$ is the set of adjacent transpositions, and that $(S_{n+1},S)$ forms a Coxeter system.

\subsubsection{Faces of the permutahedron}\label{s:faces}
When viewing the permutahedron as an $\mathcal{H}$-polytope defined as in~\eqref{e:Hpolytope}, every inequality defines a facet, and the facet $F(I)$, corresponding to the inequality for an $\emptyset \subsetneq I \subsetneq [n+1]$, has vertex set equal to $\{\,\pi \in S_{n+1} : \pi(I) = \{1, \ldots, |I| \}\,\}$.
Every facet is of this form.

Since every proper face of a polytope is an intersection of facets, it follows that the proper faces of $P(n)$ of dimension $n-k$ ($1 \leq k \leq n$) are precisely the sets of the form $F(I_1) \cap \cdots \cap F(I_k)$ for $\emptyset \subsetneq I_1 \subsetneq \cdots \subsetneq I_k \subsetneq [n+1]$.
In particular, the set of edges of $P(n)$ is
\[
  \{[\pi, \tau_i \pi] : \pi \in S_{n+1}, i \in [n]\}\text{, where $\tau_i$ denotes the transposition $(i, i+1)$,}
\]
and therefore $\P{n}$, as defined in Definition~\ref{def:permutahedron}, is the $1$-skeleton of $P(n)$.

If $\pi$ is a vertex of the face $\bigcap_{i=1}^k F(I_i)$ for $\emptyset \subsetneq I_1 \subsetneq \cdots \subsetneq I_k \subsetneq [n+1]$, then the edge $[\pi, \tau_{j}\pi]$ is contained in it if and only if $j \not\in \{|I_i| : i \in 1 \leq i \leq k\}$.
Informally, swapping two values which differ by $1$ and belong to the same `segment' $\pi(I_i \setminus I_{i-1})$ for some $1 \leq i \leq k+1$ leads to a new vertex on the same face, where $I_0 := \emptyset$ and $I_{k+1} := [n+1]$.
This implies the face $\bigcap_{i=1}^k F(I_i)$ is combinatorially equivalent to the Cartesian product of polytopes $\prod_{i=1}^{k+1} P\bigl(|I_i \setminus I_{i-1}|-1\bigr)$, which is a polytope of dimension $n-k$.
Conversely, for $0 \leq k \leq n$, each Cartesian product of $k+1$ permutahedra whose dimensions sum up to $n-k$ can be realised as a face of $P(n)$ by choosing an appropriate chain of sets $\emptyset \subsetneq I_1 \subsetneq \cdots \subsetneq I_k \subsetneq [n+1]$.

The above provides information about the $1$-skeleton of faces.
Recall that the Cartesian product of graphs $H_1, \ldots, H_k$ is the graph $\bigcart_{i=1}^k H_i$ on vertex set $\prod_{i=1}^k V(H_i)$ whose edge set is given by
\[E\bigg(\bigcart_{i=1}^k H_i\bigg) = \Big\{\,vw \,:\, \text{for some }i \in [k],\, v_i w_i \in E(H_i)\text{ and }v_j=w_j\text{ for all }j \in [k]\setminus\{i\}\,\Big\},\]
By the above, for any $\emptyset \subsetneq I_1 \subsetneq \cdots \subsetneq I_k \subsetneq [n+1]$, the $1$-skeleton of the $(n-k)$-dimensional face $\bigcap_{i=1}^k F(I_i)$ is isomorphic to the Cartesian product $\bigcart_{i=1}^{k+1} \operatorname{Perm}\bigl(|I_i \setminus I_{i-1}|-1\bigr)$, whose dimensions sum up to $n-k$.

If a graph $H$ is isomorphic to the Cartesian product $\bigcart_{i=1}^{t} \P{n_i}$ of some family of permutahedra, we will call $H$ a \emph{face graph} of dimension $d:=d(H)=\sum_{i=1}^t n_i$.
Note that such a graph is $d$-regular.
Observe also that $d$-dimensional hypercubes are face graphs, since they are the Cartesian product of $d$ copies of $K_2 = \P{1}$.
For this reason, an $n$-dimensional permutahedron has a face which is a hypercube of dimension $k = \lfloor (n+1)/2\rfloor$, corresponding to any choice of a maximal chain $\emptyset \subsetneq I_1 \subsetneq \cdots \subsetneq I_k \subsetneq [n+1]$ satisfying $|I_i \setminus I_{i-1}|=2$ for each $i \in [k]$.

\subsection{Auxiliary Lemmas}
We use the following result which bounds the number of $m$ vertex trees in a graph.

\begin{lemma}[{\cite[Lemma 2]{BFM98}}]\label{l:treecount}
  Let $G$ be a graph on $n$ vertices with maximum degree $\Delta$ and minimum degree $\delta$.
  Let $t_m(G)$ be the number of $m$-vertex rooted trees in $G$.
  Then,
  \[
   n \cdot \frac{m^{m-2}(\delta-m)^{m-1}}{(m-1)!} \le t_m(G)\le n(e\Delta)^{m-1}.
  \]
\end{lemma}

We will also need the following two results, which bound the likely diameter and maximum degree in a random $m$-vertex tree, due to R\'{e}nyi and Szekeres~\cite{RS67} and Moon~\cite{M68}, respectively.

\begin{lemma}[\cite{RS67}]\label{l:treediameter}
  Let $T$ be a tree chosen uniformly at random from all trees with vertex set $[m]$.
  Then $T$ has diameter $O_p(\sqrt{m})$.
  That is, for any $\varepsilon >0$ there exists $K>0$ such that the probability that $T$ has diameter greater than $K \sqrt{m}$ is at most $\varepsilon$ for all sufficiently large $m$.
\end{lemma}

\begin{lemma}[\cite{M68}]\label{l:treemaxdeg}
  Let $T$ be a tree chosen uniformly at random from all trees with vertex set $[m]$.
  Then $\Delta(T) = (1+o(1))\log m/\!\log \log m$ with probability tending to one as $m \to \infty$.
\end{lemma}

We will also need some basic facts about Galton--Watson trees.
\begin{lemma}\label{l:branching}
  Let $c > 1$ be a constant and let $T$ be a Galton--Watson tree with child distribution $\Bin(n,p)$ with $np = c$ and let $\gamma(c)$ be defined as in~\eqref{e:survival-prob}.
  Then
  \begin{enumerate}[$(i)$]
  \item\label{i:branching:continuous} $\gamma(c)$ is an increasing, continuous function on $(1,\infty)$;
  \item\label{i:branching:quant} $\gamma(c) > c-1$ for all $1 < c \leq 5/4$;%, and furthermore $\gamma(c) = 2(c-1) + O((c-1)^2)$ as $c \to 1^+$;
  \item\label{i:branching:survival} if $k$ is growing with $n$, then the probability that $T$ is finite and has size at least $k$ is $o_n(1)$;
  \item\label{i:branching:infinite} as $n$ tends to infinity, the probability that $T$ is infinite tends to $\gamma(c)$; and
  \item\label{i:branching:conditioning} if we condition on the fact that $T$ is infinite, then as $k$ tends to infinity
    \[
      |T_k|^{1/k} \to c
    \]
    in probability, where $|T_k|$ is the size of the $k$th generation.
  \end{enumerate}
\end{lemma}
Properties~\eqref{i:branching:continuous} and~\eqref{i:branching:quant} are simple analytic exercises given the definition~\eqref{e:survival-prob} of $\gamma(c)$.
Property~\eqref{i:branching:survival} would be immediate if $T$ were a Galton--Watson tree with a fixed child distribution, since $\sum_{k=1}^\infty \mathbb{P}(|T|=k) \leq 1$ is summable.
However, it is easy to see that, since $\Bin\bigl(n,\frac{c}{n}\bigr) \to \operatorname{Po}(c)$ in distribution as $n \to \infty$, then $\mathbb{P}(|T|=k) \to \mathbb{P}(|\hat{T}| =k)$ as $n \to \infty$ for any fixed $k$, where $\hat{T}$ is a Galton--Watson tree with child distribution $\operatorname{Po}(c)$, and so Property~\eqref{i:branching:survival} follows.
Properties~\eqref{i:branching:infinite} and~\eqref{i:branching:conditioning} are fundamental facts about branching processes, see, for example,~\cite[Theorem 1]{AN72} and~\cite[Proposition 6.4]{L90a}, respectively.

We will use two concentration bounds.
The first is a version of the Chernoff bound.
\begin{lemma}\label{l:chernoff}
  Let $n \in \mathbb{N}$, let $p \in [0,1]$ and let $X \sim \Bin(n,p)$.
  \begin{enumerate}[(i)]
  \item\label{i:chernoff1} For every positive $t$ with $t \leq \frac{3np}{2}$,
    \[
      \mathbb{P}\bigl(|X -np| \geq t\bigr) \leq 2 \exp\Bigl(-\frac{t^2}{3np} \Bigr).
    \]
  \item\label{i:chernoff2} For every positive $b$,
    \[
      \mathbb{P}\bigl(X \geq bnp \bigr) \leq \Bigl(\frac{e}{b}\Bigr)^{b np}.
    \]
  \end{enumerate}
\end{lemma}
\begin{proof}
  The first bound is~\cite[Corollary 2.3]{JLR00}.
  The second is trivial if $b \leq e$; if not, a union bound over subsets of $[n]$ of size $\lceil bnp \rceil$ gives $\mathbb{P}\bigl(X \geq bnp \bigr) \leq \binom{n}{\lceil bnp \rceil}p^{\lceil bnp \rceil} \leq \bigl(\frac{enp}{\lceil bnp \rceil}\bigr)^{\lceil bnp \rceil} \leq \bigl(\frac{e}{b}\bigr)^{bnp}$.
\end{proof}

The second is a version of the Azuma--Hoeffding inequality~\cite[Chapter 7]{AS16}.
\begin{lemma}\label{l:azuma}
  Let $X = (X_1,X_2,\ldots, X_n)$ be a random vector with independent entries and range $\Lambda = \prod_{i \in [n]} \Lambda_i$ and let $f\colon \Lambda\to\mathbb{R}$ and $C \in \mathbb{R}$ such that for every $x,x' \in \Lambda$ which differ only in one coordinate,
  \[
    |f(x)-f(x')|\le C.
  \]
  Then, for every $t\ge 0$,
  \[
    \mathbb{P}\Bigl(\bigl|f(X)-\mathbb{E}[f(X)]\bigr|\ge t\Bigr)\le 2\exp\Bigl(-\frac{t^2}{2C^2n}\Bigr).
  \]
\end{lemma}

We will also use the following inequality, which links the isoperimetric properties of a graph to the spectrum of its adjacency matrix, which is sometimes referred to as \emph{Cheeger's inequality}.
The version below is implied by work of Alon and Milman~\cite{AM85}, see~\cite[Theorem 2.4]{HLW06} for an explicit statement.
\begin{lemma}\label{l:cheeger}
  Let $G$ be an $n$-regular graph with adjacency matrix $A$.
  Then
  \[
    \iso(G) \geq  \lambda_1/2,
  \]
  where $\lambda_1$ is the second-smallest eigenvalue of the \emph{Laplacian} $nI - A$.
\end{lemma}

\section{The isoperimetric problem}\label{s:iso}

In this section we will discuss the (edge-)isoperimetric problem on $\P{n}$, showing in particular how Proposition~\ref{p:edge-iso-general} follows from results in the literature.

The first bound in Proposition~\ref{p:edge-iso-general} is a relatively straightforward consequence of Cheeger's inequality (Lemma~\ref{l:cheeger}) together with known bounds on the spectrum of $\P{n}$.
\begin{lemma}\label{l:isolargeset}
  The edge-isoperimetric constant of $\P{n}$ satisfies
  \[
    \iso(\P{n}) = \Omega\Bigl(\frac{1}{n^2}\Bigr).
  \]
\end{lemma}

\begin{proof}
  Bacher~\cite{B94} showed that the smallest positive eigenvalue $\lambda_1$ of the Laplacian of $\P{n}$ is
  \[
    \lambda_1(\P{n}) = 2 - 2\cos\Bigl(\frac{\pi}{n+1}\Bigr) = \Theta\Bigl(\frac{1}{n^2}\Bigr),
  \]
  and so the result follows from Lemma~\ref{l:cheeger}.
\end{proof}

We remark that it seems unlikely that the bound in Lemma~\ref{l:isolargeset} is optimal in terms of its dependence on $n$.
However, for our application we only require that $i(\P{n})$ is not shrinking very fast as a function of $n$, and in fact any inverse polynomial bound here would be sufficient.
We note that \emph{some} inverse polynomial factor here is necessary, as the set
\[
  S = \Bigl\{\, \sigma \in S_{n+1} : \sigma(1) \leq \frac{n+1}{2} \,\Bigr\}
\]
of size $\frac{(n+1)!}{2}$ has $\partial(S) = \frac{2 |S|}{n+1}$ and so witnesses that $i(\P{n}) \leq \frac{2}{n+1}$.
Indeed, a vertex is incident to an edge in the boundary of $S$ if and only if $\sigma(1) = \frac{n+1}{2}$, and in this case the edge is unique, corresponding to the transposition $\tau_{(n+1)/2}$.
Hence, $\partial(S) = n! = \frac{2|S|}{n+1}$.

For small sets however, Lemma~\ref{l:isolargeset} is far from the truth.
In the case of the hypercube, a classic result of Harper~\cite{H64} determines the value of $i_k(G)$ for all $k$.
In particular, Harper's result implies the following bound.
\begin{theorem}[\cite{H64}, see also~\cite{L64, B67, H76}]\label{t:Harper}
  Let $n \in \mathbb{N}$.
  For every $k\in[2^n]$,
  \[
    i_k(Q^n) \geq n-\log_2 k.
  \]
\end{theorem}
More generally, it is known that many other \emph{high-dimensional} graphs have quantitatively similar expansion properties, for example Cartesian products~\cite{DEKK23+,T00} or abelian Cayley graphs~\cite{L15}, where $i_k(G) = \Omega\bigl(\log \bigl(|G|\big/ k\bigr)\bigr)$.

The same is true of $\P{n}$, which in fact satisfies the same isoperimetric inequality (Theorem~\ref{t:Harper}) as the hypercube (see Proposition~\ref{p:edge-iso-general}), and so also has almost optimal small-set expansion.
This follows immediately from Theorem~\ref{t:Harper} together with that fact that $\P{n}$ is an induced subgraph of a hypercube (of some dimension).
Indeed, since $Q^n$ is $n$-regular, Theorem~\ref{t:Harper} is equivalent to the fact that the following holds in $Q^n$:
\begin{equation}
\begin{split}
&\text{Every subgraph of order $k$ has average degree at most $\log_2 k$.} \label{e:dimensionlessharper}
\end{split}
\end{equation}
Note that~\eqref{e:dimensionlessharper} is \emph{dimensionless}, in that it does not depend on the dimension $n$ of the ambient hypercube.
Then, since $\P{n}$ is a subgraph of some hypercube, it follows that~\eqref{e:dimensionlessharper} also holds in $\P{n}$ and so, since $\P{n}$ is also $n$-regular, the second part of Proposition~\ref{p:edge-iso-general} follows.

The most straightforward way to see that $\P{n}$ is a subgraph of $Q^N$ for some $N$ would be to consider the map from $V(\P{n})$ to the hypercube $Q^{\binom{n+1}{2}}$ whose vertices encode subsets of $\binom{[n+1]}{2}$ given by
\[
  \pi \mapsto \bigl\{\, \{a, b\} \subseteq [n+1]\, :\, (b - a) \cdot (\pi(b) - \pi(a)) < 0 \,\bigr\},
\]
which can be seen to be a graph isometry, see~\cite[Section 4]{CEGK25} for details.
More generally, Lemma~\ref{l:zonotopepartialcube} shows that the $1$-skeleton of \emph{any} zonotope is not only a subgraph of a hypercube, but moreover a partial cube.

Let us discuss briefly the edge-isoperimetric problem in $\P{n}$ more generally.
Since $\P{n}$ is $n$-regular, it is clear that Harper's inequality is asymptotically tight when $\log_2 k = o(n)$.
In fact, it is tight for all $k$ of the form $2^r$ with $r \leq \frac{n+1}{2}$.
Indeed, for all such $k$, $P(n)$ contains a face $F$ which is a hypercube of dimension $r$, and since $\P{n}$ is $n$-regular, it follows that $|\partial(F)| = |F|(n-r) = |F|(n-\log_2 k)$.

Here then, as opposed to the case of the hypercube, we see that the optimal sets are not given by lower dimensional permutahedra, as perhaps might be expected at first glance.
Indeed, subsets of the form $\P{r} \subseteq \P{n}$ of size $k=(r+1)!$ have an expansion factor of $n-r \approx n- \frac{\log k}{\log \log k}$, which is significantly larger than the bound given by Theorem~\ref{t:Harper}.

The point here is that $\P{n}$ contains faces of much higher density than $\P{r}$, and in particular the densest faces are copies of the hypercube, the largest of which has dimension $\frac{n+1}{2}$.
It seems likely that, at least for $k \leq 2^{\frac{n+1}{2}}$, the optimal sets for the edge-isoperimetric problem are precisely those which optimise the edge boundary inside a copy of the hypercube, and so in particular can be chosen to be nested.

For larger $k$ it becomes less clear what structure the minimising sets should have.
If we restrict ourselves to looking at subsets of faces, then the expansion ratio `outside' the face is smallest when the face is as dense as possible.
So, given $k\in \mathbb{N}$ we might want to choose a face $F$ of size at least $k$, which is as dense as possible given this restriction, and choose a subset of a face $F$ of size $k$ which minimises the boundary inside $F$.

Since each face of dimension $d$ induces a regular subgraph of degree $d$, and each face is a product of lower dimensional permutahedra, it is relatively clear that for a fixed dimension $d$, the largest faces are those in which the dimension of the factors are as evenly split as possible.
In particular, if $k=6^{\frac{n+1}{3}}$ this heuristic would suggest an optimal set is a face which arises as the Cartesian product of $\frac{n+1}{3}$ copies of $\P{2}$ (that is, the $6$-cycle), which has an edge-boundary of order $k\bigl(\frac{n+1}{3} -1 \bigr)$.

\begin{conjecture}
  Let $k = 6^{\frac{n+1}{3}}$.
  Then $i_k\bigl(\P{n}\bigr) = \frac{n+1}{3}-1$.
\end{conjecture}

However, since the largest proper face in $\P{n}$ has order $n!$, this intuition cannot extend to larger subsets of $\P{n}$, for example sets of size $\Omega\bigl((n+1)!\bigr)$.
By Lemma~\ref{l:isolargeset} and the discussion following it, we know that
\[
  \frac{2}{n+1} \geq i\bigl(\P{n}\bigr) = \Omega\Bigl(\frac{1}{n^2} \Bigr),
\]
and it would be very interesting to determine what the correct order polynomial dependence on $n$ is, and if the optimal sets in the range from $n!$ to $(n+1)!$ are also given, as in the example, by appropriate unions of faces.
\begin{question}
  What is the minimal $c \in \mathbb{R}$ such that
  \[
    i\bigl(\P{n}\bigr) = \Omega\bigl(n^{-c}\bigr)?
  \]
\end{question}

\section{The projection lemma}\label{s:projection}
A key property of the permutahedron that will form the basis of our analysis is that it has a certain type of fractal self-symmetry.
Roughly, given a small set $X$ of vertices, we can cover $X$ with a family of disjoint subgraphs each of which is in some way a lower-dimensional graph with similar properties to the permutahedra.
Similar results have been key to the analysis of percolation in the hypercube and other high-dimensional product graphs~\cite{CEGK24,DEKK22,DEKK23+,DK23}.

A key difference here is that we cannot necessarily cover the set $X$ with disjoint copies of lower dimensional permutahedra, but must allow ourselves more freedom in what we take as our projections, working instead with the more general concept of a face graph.
Recall from Subsection~\ref{s:faces} that a graph $H$ isomorphic to the Cartesian product $\bigcart_{i=1}^{t} \P{n_i}$ of some family of permutahedra is called a face graph of dimension $d(H)=\sum_{i=1}^t n_i$.
\begin{lemma}[Projection Lemma]\label{l:proj}
  Let $H$ be a face graph of dimension $m$ and let $X \subseteq V(H)$ have size $|X|=k$.
  Then there is a disjoint family of subgraphs $\{H(x) : x \in X\}$ of $H$ such that
  \begin{itemize}
  \item Each $H(x)$ is a face graph of dimension at least $m+1-k$;
  \item $x \in V(H(x))$ for all $x \in X$.
  \end{itemize}
\end{lemma}

We will often refer to the face graphs guaranteed by Lemma~\ref{l:proj} as \emph{projections} of the graph $H$.

We will see that Lemma~\ref{l:proj} follows from the following lemma using the representation of the permutahedron as a Cayley graph of a Coxeter group.

\begin{lemma}\label{l:coxeterproj}
  Let $(\Gamma,\{r_1,\ldots,r_m\})$ be a Coxeter system and let $S \subseteq \Gamma$ have size $|S|=s$.
  Then there is a family of subsets $\{I(x) : x \in S\}$ of $[m]$ and a disjoint family of cosets $\{w(x)\Gamma_{I(x)} : x \in S\}$ such that
  \begin{itemize}
  \item $|I(x)|\leq s-1$ for all $x \in S$;
  \item $x \in w(x)\Gamma_{I(x)}$ for all $x \in S$.
  \end{itemize}
\end{lemma}
\begin{proof}
  We induct on $s$, where for any $m$ the case $s=1$ is trivial (taking $I(x) = \emptyset$).
  For each $i \in [m]$, the cosets of $\Gamma_i$ partition $\Gamma$ and by Lemma~\ref{l:Coxeter} the intersection of any family of cosets $\bigcap_{i=1}^m w_i \Gamma_i$ has size at most one.
  It follows that there must be some $i$ such that the partition of $S$ induced by the cosets of $\Gamma_i$ is non-trivial.
  However, since $\Gamma_i$ is isomorphic to the Coxeter system $\bigl(\Gamma_i,\{r_1,\ldots, r_{i-1}, r_{i+1}, \ldots, r_m\}\bigr)$, and each coset contains at most $s-1$ points, the claim follows by induction.
\end{proof}

To see that Lemma~\ref{l:proj} follows from Lemma~\ref{l:coxeterproj}, we note first that $\P{n}$ is the Cayley graph of $S_{n+1}$ with respect to the generating set $S = \{\tau_i : i \in [n]\}$.
Furthermore, if we define for each $X \subseteq S$ the subgroup $H(X)$ of $S_{n+1}$ which is generated by $S \setminus X$, then there is a one-to-one correspondence between faces of $\P{n}$ of codimension $k$ and cosets of the subgroups $\{H(X) : |X|=k\}$.
In particular, if a face $F$ corresponds to a coset of such a subgroup $H(X)$, then the $1$-skeleton of $F$ is isomorphic to the Cayley graph of $H(X)$ with respect to the generating set $S \setminus X$.
Hence, since $(H(X),S \setminus X)$ is a Coxeter system for each $X \subseteq S$, Lemma~\ref{l:proj} follows from Lemma~\ref{l:coxeterproj}.

We also give an alternative proof using the representation of the permutahedron as a zonotope.
To do this, we prove a projection lemma which holds for all simple zonotopes.

\begin{lemma}\label{l:zonoproj}
  Let $Z$ be a simple $d$-dimensional zonotope and let $S \subseteq V(Z)$ have size $|S| = s$.
  Then there exists a disjoint family of faces $\{F(x) : x \in S\}$ such that
  \begin{itemize}
  \item $F(x)$ has dimension at least $d-s+1$ for all $x \in S$;
  \item $x \in F(x)$ for all $x \in S$.
  \end{itemize}
\end{lemma}
\begin{proof}
  By translating $Z$ if necessary, we may assume $Z=\sum_{i=1}^k [-\bm{v_i},\bm{v_i}]$ with all $\bm{v_i}$ nonzero and no two $\bm{v_i}$ parallel.
  It follows from the proof of Lemma~\ref{l:zonotopepartialcube} that every $x \in V(Z)$ has a unique representation of the form
  \[
    x = \sum_{i=1}^k \eps_i(x) \bm{v_i}
  \]
  for $(\eps_1(x),\ldots,\eps_k(x)) = \pi^{-1}(x)$, where $\pi$ is the projection map whose matrix has columns $\bm{v_1}, \ldots, \bm{v_k}$.
  Therefore, there is a set $J(x) \subseteq [k]$ of size $d$ such that $[x, x - 2\eps_j(x)\bm{v_j}]$ is an edge of $Z$ if and only if $j \in J(x)$.
  Since $Z$ is simple, the vertex figure at $x$ is a simplex, implying the sets
  \[
    G_T(x) = Z \cap \operatorname{aff}\bigl(\{x\} \cup \{x -2\eps_j(x)\bm{v_j} : j \in T\}\bigr)
  \]
  are all distinct faces for $T \subseteq J(x)$, where $\operatorname{aff}(X)$ denotes the affine span of a set $X$.
  Moreover, $G_T(x)$ has dimension $|T|$, and $\pi^{-1}(G_{T}(x))$ is the subcube of $Q^k$ containing $\pi^{-1}(x)$ whose set of free coordinates, the coordinates in the subcube which may vary, is
  \[
    C_T(x) = \bigl\{\,i \in [k] : \bm{v_i} \in \operatorname{span}\bigl(\{\bm{v_j} : j \in T\}\bigr)\,\bigr\}.
  \]
  The sets $C_T(x)$ satisfy the following claim.
  \begin{claim*}
    For every $x \in V(Z)$, $T \subseteq J(x)$ and $i \in C_T(x)$, there exists $j \in T$ such that $i \not\in C_{T\setminus \{j\}}(x)$.
  \end{claim*}
  \begin{proof}[Proof of claim]
    Since $Z$ is simple, $\{\bm{v_j} : j \in J(x)\}$ is a linearly independent set of vectors, and therefore $\bigcap_{j \in T} \operatorname{span}\bigl(\{\bm{v_t} : t \in T \setminus \{j\}\}\bigr) = \{0\}$.
    In particular, we have that $\bigcap_{j \in T} C_{T\setminus \{j\}}(x) = \emptyset$, from where the stated property follows.
  \end{proof}

  We now run the following algorithm: For each $x \in S$, initialise $J'(x)$ to $J(x)$.
  Now repeatedly do the following: If there are distinct $x, y \in S$ such that $\pi^{-1}(G_{J'(x)}(x))$ and $\pi^{-1}(G_{J'(y)}(y))$ intersect, let $i \in [k]$ be a coordinate in which $\pi^{-1}(x)$ and $\pi^{-1}(y)$ differ.
  For each $z \in \{x, y\}$ such that $i \in C_{J'(z)}(z)$, apply the claim to obtain $j_z \in J'(z)$ such that $i \not\in C_{J'(z) \setminus \{j_z\}}(z)$, and update $J'(z)$ by removing $j_z$.
  Now $\pi^{-1}(G_{J'(x)}(x))$ and $\pi^{-1}(G_{J'(y)}(y))$ are disjoint, since the $i$th coordinate is not free in either subcube and $\pi^{-1}(x)$ and $\pi^{-1}(y)$ differ in the $i$th entry.

  The procedure stops when $\{\pi^{-1}(G_{J'(x)}(x)) : x \in S\}$ is a family of pairwise disjoint subcubes.
  A given vertex $x$ may participate in at most $s-1$ iterations of the above algorithm, since once $\pi^{-1}(G_{J'(x)}(x))$ and $\pi^{-1}(G_{J'(y)}(y))$ are disjoint, they remain disjoint throughout the rest of the process.
  Since each operation reduces $|J'(x)|$ by at most $1$, it holds that $|J'(x)| \geq d - s + 1$ for every $x \in S$.
  Therefore, the faces $F(x) := G_{J'(x)}(x)$ all have dimension at least $d-s+1$.
  By construction, the subcubes $\{\pi^{-1}(F(x)) : x \in S\}$ are pairwise disjoint, which implies the faces $\{F(x) : x \in S\}$ are too.
  This finishes the proof.
\end{proof}

It can be shown (see for example~\cite[Section 2.3]{ALVSWZ99}) that every reflection arrangement is simplicial, and so the associated zonotope~\cite[Corollary 7.17]{Z95} is simple, and its $1$-skeleton is isomorphic to the Cayley graph of the Coxeter group corresponding to the reflection arrangement.
In this way, Lemma~\ref{l:zonoproj} implies Lemma~\ref{l:coxeterproj}.
Moreover, every face of a simple polytope is also simple, and every face of a zonotope is also a zonotope.
Since the permutahedron is a simple zonotope, Lemma~\ref{l:zonoproj} directly implies Lemma~\ref{l:proj}.

We suspect that the restriction to simple zonotopes in Lemma~\ref{l:zonoproj} is not necessary, at least in a qualitative manner, and a similar statement, with perhaps a worse bound on the co-dimension of the faces, should hold for general zonotopes.
More generally, it would be interesting to investigate the following quantity for other classes of polytopes: Given a polytope $Q$, let $f(Q,s)$ be the largest $k$ such that for any subset $S \subseteq V(Q)$ of size $s$ there is a disjoint family of faces $\{F(x) \colon x \in S\}$ such that $x \in F(x)$ for all $x\in S$ and each $F(x)$ has dimension at least $k$.

For example, when we choose $Q$ to be either the $n$-dimensional hypercube or the $n$-dimensional permutahedron, it follows (for example from Lemma~\ref{l:zonoproj}) that $f(Q,s) \geq n-s+1$, and it is easy to see that this is tight.
Hence, in these cases $f$ decreases linearly with $s$.
On the other hand, for the $n$-dimensional simplex $\Delta(n)$ it is easy to verify that $f(\Delta(n),s) = \lfloor \frac{n+1}{s}\rfloor -1$, which decreases much faster.
Due to the combinatorial applications of Lemma~\ref{l:proj}, it would be interesting to find other classes of polytopes for which $f$ decreases slowly, for example linearly, as a function of $s$.

\section{Projection-first search}\label{s:PFS}
A classic technique in the study of percolated subgraphs is to explore the component (cluster) $C_v$ in $G_p$ containing a vertex $v$ using breadth-first search (BFS), where the `existence' or `non-existence' of an edge $e \in E(G)$ in $G_p$ is exposed dynamically as they are processed in the algorithm.
When the host graph $G$ is $n$-regular, then, at least in the early stages, before we have discovered $\Theta(n)$ many vertices, this process behaves like a Galton--Watson branching process with child distribution $\Bin(n,p)$.
In this way, we can relate the probability that $v$ lies in a `large' cluster to the survival probability of this branching process.

In $G(n,p)$ this intuition can be used to show that in the supercritical regime whp the correct proportion of vertices lie in components of linear order, and a simple sprinkling argument shows that these linear components all merge into a unique \emph{giant component}.
However, when the order of the host graph is much larger than its regularity, as in many high-dimensional geometric graphs such as the hypercube or permutahedron, we cannot necessarily guarantee that the clusters will grow larger than $\Theta(n)$ with positive probability in this manner, which causes difficulties in the merging step.

In this section, we describe a variant of the BFS algorithm, which we call projection-first search (PFS), which grows significantly larger percolation clusters.
The idea here is to avoid `backtracking' by carefully separating the vertices of the BFS tree into disjoint parts using the projection lemma.
In this way, we can guarantee that the number of explored neighbours of a vertex will decrease much more slowly during the process: linearly in the depth of the BFS tree, rather than linearly in the size of the tree.
This will ensure that the process can be approximated by a $\Bin(n,p)$ branching process until the tree grows to \emph{depth} $\Theta(n)$, rather than size $\Theta(n)$.
Since, in the supercritical regime, we expect the layers of the BFS tree to be growing exponentially quickly, we should then be able to build exponentially large clusters in this manner.

\subsection{The algorithm}
The procedure takes as input some face graph $H$ of dimension $m$.
We will keep track of a number of different subsets and subgraphs of $H$.
The algorithm will run in a number of rounds, and in the $t$th round of the algorithm we will have
\begin{itemize}
\item A set $W(t) \subseteq V(H)$ of \emph{explored} vertices, such that all edges exposed in rounds $1$ to $t$ (inclusive) are incident to $W(t)$;
\item A set $U(t) = V(H) \setminus W(t)$ of \emph{unexplored} vertices;
\item A \emph{frontier} $A(t) \subseteq W(t)$, such that no edge from $A(t)$ to $U(t)$ has been exposed before the $t$th round;
\item A subgraph $T(t) \subseteq H_p$;
\item A vertex-disjoint family $\{ H(x) : x \in A(t)\}$ of face subgraphs of $H$, such that $H(x) \cap W(t) = \{x\}$ for all $x \in A(t)$.
\end{itemize}
We initialise by setting $W(1) = A(1) =\{v\}$, $U(1) = V(H) \setminus \{v\}$, $T(1) = (W(1),\varnothing)$, and $H(v) = H$.
For convenience, we also set $W(0) = \emptyset$.

In the $t$th round, for each vertex $x \in A(t)$ we expose the neighbours $N_t(x)$ of $x$ in $H(x)_p \subseteq H_p$\footnote{Here, in an abuse of notation, given a subgraph $H' \subseteq H$ we write $H'_p$ for the subgraph $H_p \cap H'$.}.
Note that since $H(x) \cap W(t) = \{x\}$, the edges incident to $x$ in $H(x)$ have not yet been exposed in the exploration process.
We then apply Lemma~\ref{l:proj} to the set $N_t(x) \cup \{x\}$ inside the face graph $H(x)$ to obtain a disjoint family of projections, and for each $y \in N_t(x)$ we set $H(y)$ to be the projection which contains $y$, noting that each $H(y)$ has dimension at least $d(H(x)) - |N_t(x)|$.
We also set
{\allowdisplaybreaks\begin{align*}
  W(t+1) &= W(t) \cup N_t(A(t)),\\
  U(t+1) &= U(t) \setminus N_t(A(t)),\\
  A(t+1) &= N_t(A(t)),\\
  T(t+1) &= \bigl(W(t+1),\; E(T(t)) \cup \bigl\{xw : x \in A(t)\text{ and }w \in N_t(x)\bigr\}\bigr),
\end{align*}}%
where $N_t(A(t)) = \bigcup_{x\in A(t)} N_t(x)$.
Note that, since the family $\{H(x) : x \in A(t)\}$ is disjoint by assumption and for all $y \in N_t(x)$, $H(y) \subseteq H(x)$, it then follows that $\{H(y) : y \in A(t+1)\}$ is a disjoint family by construction.
Furthermore, for any $x \in A(t)$ and any $y \in N_t(x)$, since $H(y) \subseteq H(x)$ and $H(y) \cap  (N_t(x) \cup \{x\}) = \{y\}$, it follows that $H(y) \cap W(t+1) = \{y\}$.

Let us note some basic facts about the algorithm, which hold for each $t$:
\begin{enumerate}[(a)]
\item $T(t)$ is a tree whose vertex set is $W(t)$;
\item\label{i:disjointunion} $W(t) = \bigcup_{i=1}^t A(i)$, and this union is disjoint;
\item\label{i:neighbourhooddist} For each $x \in A(t)$, $|N_t(x)|$ is distributed as a binomial random variable $\Bin(d(H(x)),p)$;
\item\label{i:dimension} For each $x \in A(t)$, the dimension of $H(x)$ is at least $m - w(x)$ where, writing $v=v_1v_2\ldots v_{t}=x$ for the path $vTx$ from $v$ to $x$ in $T$,
  \begin{equation}\label{e:dimensionreduction}
    w(x):=\sum_{i=1}^{t-1} |N_i(v_i)| = d_T(v_1) + \sum_{i=2}^{t-1} (d_T(v_i) -1).
  \end{equation}
\end{enumerate}

\subsection{Cluster sizes}

Analysing the projection-first search, we can show large clusters appear in the percolated face graph $H_p$ with the `correct' probability.
The lemmas in this section will later be applied with $H = \P{n}$.

\begin{lemma}\label{l:PFS}
  Let $H$ be a face graph of dimension $m\in \mathbb{N}$, let $\beta >0$ and let $p\geq\frac{1+\beta}{m}$.
  Then, for all $v\in V(H)$,
  \[
    \mathbb{P}\biggl(|C_v| = \exp\Bigl(\Omega\Bigl(\frac{m}{\log m} \Bigr)\Bigr)\biggr) \geq \gamma(1+\beta) +o_m(1).
  \]
\end{lemma}

\begin{proof}
  Let $v \in V(H)$.
  We first note that, since the event that $|C_v|$ is large is an increasing event, we may assume that $p = \frac{1+\beta}{m}$. Let $f(m) = \sqrt{\log \log m}$, so that $f(m) = \omega(1)$ and $f(m) = o(\log\log m)$ (asymptotics in this proof are as $m \to \infty$).
  Suppose that we run the PFS algorithm starting at $v$ for $\tau=\frac{m}{\log m}$ rounds, conditioning at the start of each round $t\leq \tau$ on the event
  \[
    \mathcal{E}(t) = \text{``$w(x) \leq \frac{m}{f(m)}$ for all $x \in W(t)$''}.
  \]

  Due to the conditioning, for each $t \leq \tau$ and every $x \in A(t)$ the dimension of $H(x)$ is at least $\bigl(1-\frac{1}{f(m)}\bigr)m$ by fact~\eqref{i:dimension}.
  Hence, for each $t \leq \tau$ and each $x \in A(t)$, the size of the neighbourhood $N_t(x)$ stochastically dominates a $\Bin\bigl(\bigl(1-\frac{1}{f(m)}\bigr)m,p\bigr)$ random variable, which has expectation
  \[
    \Bigl(1-\frac{1}{f(m)}\Bigr)mp \geq 1+\beta + o(1).
  \]
  Since the face subgraphs $H(x)$ are disjoint for different $x \in A(t)$, and the event $\mathcal{E}(t)$ only depends on edges which are incident to $W(t-1)$, it follows that, under this conditioning, we can couple the exploration process from below with a Galton--Watson tree with a binomial child distribution with expectation at least $1+\beta + o(1)$.
  By Lemma~\ref{l:branching}~\eqref{i:branching:infinite} and~\eqref{i:branching:conditioning}, by the $\tau$th round this process grows to size
  \[
    \bigl(1+\beta + o(1)\bigr)^\tau = \exp\Bigl(\Omega\Bigl(\frac{m}{\log m} \Bigr)\Bigr)
  \]
  with probability at least $\gamma(1+\beta) + o(1)$.
  Therefore,
  \begin{equation}\label{e:pfs-bound}
    \mathbb{P}\biggl(|W(\tau+1)| = \exp\Bigl(\Omega\Bigl(\frac{m}{\log m} \Bigr)\Bigr)\biggm| \mathcal{E}(\tau)\biggr) \geq \gamma(1+\beta) + o(1).
  \end{equation}
  We claim that $\mathbb{P}(\mathcal{E}(\tau)^c) = o(1)$, which together with~\eqref{e:pfs-bound} implies the desired result.
  Note that $W(1) = \{v\}$ and $w(v) = 0$, so $\mathcal{E}(1)$ always holds.
  By considering the earliest failing round and applying a union bound, we have
  \begin{equation}\label{e:pfs-e-union-bound}
    \mathbb{P}\bigl(\mathcal{E}(\tau)^c\bigr) = \sum_{t=1}^{\tau-1} \mathbb{P}\bigl(\mathcal{E}(t) \cap \mathcal{E}(t+1)^c \bigr) \leq \sum_{t=1}^{\tau-1}\sum_{y \in A(t+1)} \mathbb{P}\Bigl(\mathcal{E}(t) \cap w(y) > \frac{m}{f(m)} \Bigr).
  \end{equation}

  For each $t \in [\tau]$ and $y \in A(t+1)$, observe that, by~\eqref{e:dimensionreduction}, $w(y)$ is a sum of $t$ independent random variables, each stochastically dominated by a $\Bin(m,p)$ random variable by fact~\eqref{i:neighbourhooddist}. This lets us apply the Chernoff bound (Lemma~\ref{l:chernoff}~\eqref{i:chernoff2}) and conclude that every $y \in A(t+1)$ satisfies
  \[
    \mathbb{P}\Bigl(\mathcal{E}(t) \cap w(y) > \frac{m}{f(m)} \Bigr) \leq \mathbb{P}\Bigl(w(y) > \frac{m}{f(m)} \Bigr) \leq \Bigl( \frac{e(1+\beta)t f(m) }{m}\Bigr)^{m/f(m)},
  \]
  which is $\exp\bigl(-\omega(m)\bigr)$ for every $t \leq \tau$ since $\log \log m = \omega(f(m))$.
  Together with the deterministic bound $|A(t+1)| \leq m^{t} \leq m^\tau = e^m$, this implies the right-hand side of~\eqref{e:pfs-e-union-bound} tends to zero, proving the claim and finishing the proof.
\end{proof}

In the above, we have focused on giving a particularly straightforward proof, emphasising the natural coupling with a branching process, without attempting to optimise the size of the clusters, and it is natural to ask about the limit of these methods.

On a heuristic level, each time we project we expect the dimension of the projection to decrease by a constant, and so we expect the process to remain supercritical until we've discovered $\Theta(m)$ many layers.
Furthermore, whilst the process remains (strictly) supercritical, for example whilst $p \cdot d(H(x)) \geq 1+\frac{\beta}{2}$ for all $x \in A(t)$, we expect the size of the frontier to grow exponentially quickly, with rate at least $\bigl(1+\frac{\beta}{2}\bigr)$.

In particular, if we are more careful in our analysis, we might hope to build clusters of exponential size in this manner.
For example, if we naively truncate our exploration process in each round, to ensure that we only uncover at most $K$ neighbours of each vertex, so that the dimension never reduces by more than $K$ in each step, we can choose $K=K(\beta)$ large enough that this process still stochastically dominates a supercritical branching process for at least $\Theta(m)$ layers.
However, the expected number of children in the coupled branching process is then smaller, and hence the probability that the branching process percolates is slightly smaller, leading to a suboptimal bound on the probability that a vertex is contained in a large cluster.

However, if instead we first run the PFS process for $\omega(1)$ rounds, then by Lemma~\ref{l:branching} with probability at least $\gamma(1+\beta) + o(1)$ not only does the process survive, but the frontier has size $\omega(1)$.
At this point we can switch to the truncated PFS process described above, which ensures that the dimension does not drop too much between the layers, which will then remain supercritical for $\Omega(m)$ layers.
If we carefully implement this modified process, which we denote by PFS$'$, we obtain the following strengthening of Lemma~\ref{l:PFS}, whose proof we defer to Appendix~\ref{a:modified-PFS}.
\begin{lemma}\label{l:PFSprecise}
  Let $H$ be a face graph of dimension $m\in \mathbb{N}$, let $\beta >0$ and let $p\geq\frac{1+\beta}{m}$.
  Then, for all $v\in V(H)$,
  \[
    \mathbb{P}\bigl(|C_v| = e^{\Omega(m)}\bigr)\geq \gamma(1+\beta) +o_m(1).
  \]
\end{lemma}

The method of projection-first search can be applied to any class of graphs where a quantitatively similar projection lemma holds, for example the hypercube or arbitrary product graphs (see~\cite[Lemma 3.1]{DEKK22}), or various classes of bipartite Kneser graphs, such as the middle layers graph (a similar idea was used in~\cite[Lemma 8.5]{CEGK24}, see Appendix~\ref{a:kneser} for details).
A comparison of this method to previous proofs of the percolation threshold for these graph classes is also provided in Appendix~\ref{a:modified-PFS}.

\subsection{Consequences of PFS for percolation}
Given a graph $G$ and $r\in \mathbb{N}$, let us write $V_{\geq r}(G)$ for the set of vertices in $G$ contained in a component of size at least $r$.
For our application to the permutahedron it will be convenient to fix \[ r:=n^{14}\] for the rest of the paper.

The first application of Lemma~\ref{l:PFS} is to conclude that the vertices in large clusters are relatively well-distributed throughout $\P{n}$, in the sense that every vertex is within distance two of a vertex contained in a large cluster.

\begin{lemma}\label{l:dense1}
  Let $\beta >0$ and let $p\geq \frac{1+\beta}{n}$.
  Then whp every vertex in $\P{n}_p$ is within distance two (in $\P{n}$) of $V_{\geq r}(\P{n}_p)$.
\end{lemma}
\begin{proof}
  Let $s= \frac{n}{\log n}$.
  Given a vertex $v \in V(\P{n})$, let $W = \{w_1,\ldots, w_s\} \subseteq N_{\P{n}}(v)$ be an arbitrary subset of the neighbours of $v$.
  We apply Lemma~\ref{l:proj} to $W \cup \{v\}$ to obtain a family $\{H(w_i) : i \in [s]\}$ of disjoint face graphs of dimension at least $\bigl(1-\frac{1}{\log n} \bigr)n$.
  For each $w_i$, let $\{w_{i,1},\ldots, w_{i,s}\} \subseteq N_{H(w_i)}(w_i)$ be an arbitrary subset of the neighbours of $w_i$ in $H(w_i)$.
  Note that each $w_{i,j}$ is at distance two from $v$.
  Again, by Lemma~\ref{l:proj} we can find a family $\{H(w_{i,j}) : i,j\in [s]\}$ of disjoint face graphs of dimension at least $m:=\bigl(1-\frac{2}{\log n} \bigr)n$.

  For each $i,j \in [s]$, since the dimension of each $H(w_{i,j})$ is at least $m$ and $p = \frac{1+\beta}{n} \geq \frac{1+ \beta +o(1)}{m}$, by Lemma~\ref{l:PFS} the probability that $w_{i,j}$ is contained in a component of order at least $r$ in $H(w_{i,j})_p$ is at least $\gamma(1+\beta) + o(1)$, and since the $H(w_{i,j})$ are disjoint, these events are independent for different $i,j \in [s]$.
  Hence, the probability that no $w_{i,j}$ lies in $V_{\geq r}(\P{n}_p)$ is at most
  \[
    \bigl(1-\gamma(1+\beta) + o(1)\bigr)^{s^2} =  o \biggl( \frac{1}{(n+1)!} \biggr),
  \]
  and the result follows by taking a union bound.
\end{proof}

The second consequence is a slightly more technical notion of density, which will be useful later.
It implies that whp every large connected set of vertices contains many vertices with many neighbours contained in large clusters.
\begin{lemma}\label{l:dense2}
  Let $\beta >0$, let $\beta'=\min\bigl\{ \beta, \frac{1}{2}\bigr\}$, let $\alpha \leq 2^{-5}(\beta')^2$ and let $p\geq \frac{1+\beta}{n}$.
  Then whp every connected subset $M \subseteq \P{n}$ of size at most $n^{3/2}$ contains at most $\alpha n$ vertices with at most $\alpha n$ neighbours in $V_{\geq r}(\P{n}_p)$.
\end{lemma}
\begin{proof}
  Since $\P{n}$ is connected, we may enlarge $M$ if necessary and consider only connected sets $M \subseteq \P{n}$ of size exactly $m=n^{3/2}$.
  Fix such a set $M$ and a subset $X \subseteq M$ of size $\alpha n$.
  We start by applying Lemma~\ref{l:proj} to find a family $\{H(x) : x \in X\}$ of disjoint projections of dimension at least $(1-\alpha)n$ such that $x\in H(x)$ for each $x \in X$.
  We will upper bound the probability that every $x \in X$ has fewer than $\alpha n$ neighbours in $V_{\geq r}(H(x)_p)$.
  A union bound over all choices of $M$ and $X$ will then finish the proof.

  Fix a particular $x \in X$ and choose a subset $W = \{w_{1}, \ldots, w_{s}\}$ of $s=8\alpha n/\beta'$ neighbours of $x$ in $H(x)$.
  Applying Lemma~\ref{l:proj} again inside $H(x)$, we may find a family $\{H(x,i) : i\in [s]\}$ of disjoint projections of dimension at least $(1-\alpha-8\alpha/ \beta')n$ such that $w_{i} \in H(x, i)$.
  Since $\alpha + 8\alpha/\beta' \leq \beta'/2$, by Lemma~\ref{l:PFS} and Lemma~\ref{l:branching}~\eqref{i:branching:continuous} and~\eqref{i:branching:quant} each of these neighbours is contained in $V_{\geq r}(H(x,i)_p)$ independently with probability at least
  \[ \gamma\biggl(\Bigl(1+\beta\Bigr)\Bigl(1-\frac{\beta'}{2}\Bigr)\biggr) +o(1) \geq \gamma\biggl(1 + \frac{\beta'}{4} \biggr) +o(1) \geq \frac{\beta'}{4}.\]
  Therefore, we expect that at least $\frac{\beta'}{4} \cdot s = 2\alpha n$ elements of $W$ are in $V_{\geq r}(H(x)_p)$.
  By the Chernoff bound (Lemma~\ref{l:chernoff}~\eqref{i:chernoff1}), the probability that a particular $x\in X$ has at most $\alpha n$ neighbours in $V_{\geq r}(H(x)_p)$ is $\exp(-\Omega(n))$.
  Since the $H(x)$ are disjoint, the probability that every $x \in X$ has fewer $\alpha n$ neighbours in $V_{\geq r}(H(x)_p)$ is $\exp(-\Omega(n|X|))=\exp(-\Omega(n^2))$.

  It only remains to apply a union bound.
  By Lemma~\ref{l:treecount} there are at most $(n+1)!\,(en)^{m}$ connected subsets of $\P{n}$ of size $m$, and given a fixed such subset $M$ there are at most $\binom{m}{\alpha n}$ many choices for the subset $X$.
  Hence, the probability that the lemma fails to hold is at most
  \[
    (n+1)!\,(en)^{m}  \binom{m}{\alpha n}\exp\bigl(-\Omega(n^2)\bigr) = o(1),
  \]
  as claimed.
\end{proof}

Another application of the PFS algorithm, of a slightly different flavour, concerns counting small subgraphs.
In $G(n,p)$ and $Q^n_p$, a matching lower bound on the size of the largest component in the subcritical regime and the second-largest component in the supercritical regime can be given by using Lemma~\ref{l:treecount} and a second moment argument.
However, the lower bound in Lemma~\ref{l:treecount} is ineffective once we start considering trees whose order is linear in $n$.
Using PFS we can give an asymptotically matching lower bound for trees of almost quadratic size.

\begin{lemma}\label{l:permtreecount}
  Let $m=m(n)$ be such that $m = \omega(1)$ and $m = o\bigl(\bigl(\frac{n}{\log n}\bigr)^2\bigr)$ and let $t_m(\P{n})$ be the number of rooted $m$-vertex trees in $\P{n}$.
  Then,
  \[
    t_m(\P{n}) \geq (1+o(1))^m (n+1)!\, \frac{m^{m-2} n^{m-1}}{(m-1)!} = e^{o(m)} (n+1)!\, (en)^{m-1}.
  \]
\end{lemma}
\begin{proof}
  We follow the proof of~\cite[Lemma 2]{BFM98}, replacing their use of BFS with PFS\@.
  Call an $m$-vertex rooted, labelled tree \emph{typical} if its maximum degree is at most $\frac{2\log m}{\log \log m}$ and its depth is at most $\sqrt{m}\log\log m$.

  We first note that by Lemmas~\ref{l:treediameter} and~\ref{l:treemaxdeg}, almost every spanning tree of $K_m$ is typical.
  Therefore, by Cayley's formula, the family $\mathcal{T}$ of typical spanning trees on vertex set $[m]$ rooted at $m$ satisfies
  \[
    |\mathcal{T}| = (1+o(1))m^{m-2}.
  \]

  Fix a vertex $v$ of $\P{n}$.
  We count the number of pairs $(T,\phi)$ where $T \subset \P{n}$ is a typical $m$-vertex tree rooted at $v$ and $\phi \colon [m] \to V(T)$ is a bijection with $\phi(m) = v$.
  Clearly each typical tree rooted at $v$ is contained in precisely $(m-1)!$ such pairs, and each pair $(T,\phi)$ determines a typical tree $T' \subset K_m$ such that $(x,y) \in E(T')$ if and only if $(\phi(x),\phi(y)) \in E(T)$.

  We will now fix a typical tree $T'$ with vertex set $[m]$ and construct many pairs $(T,\phi)$ which determine $T'$.
  To do so, we will define the map $\phi$ by exploring the tree $T'$, starting at the root $m$, layer by layer using projection-first search.

  Suppose we have a partial embedding $\phi$ of the first $i$ layers of $T'$ into $\P{n}$, where $F_i \subseteq V(T')$ is the $i$th layer, and furthermore we have a disjoint family $\{Q(x) : x \in F_i\}$ of projections whose dimension is at least $n - w(x)$, where $w(x)$ is defined as in~\eqref{e:dimensionreduction}.
  For the initial step we can take $\phi(m) = v$ and $Q(m) = \P{n}$.

  For each vertex $y \in F_{i+1}$ which is adjacent to $x \in F_i$ in $T'$, we choose an arbitrary neighbour of $\phi(x)$ in $Q(x)$ to assign as $\phi(y)$.
  Then, for each $x \in F_i$, we apply Lemma~\ref{l:proj} to find an appropriate family of projections $\{Q(y) : y\in N_{T'}(x) \cap F_{i+1} \}$.

  At the end of this process we see that for each vertex apart from the root we had at least $n- \max_{v\in T} w(v) \geq n - 2\sqrt{m} \log m = (1+o(1))n$ choices for the embedding, and hence the total number of pairs we can build in this way is at least
  \[
    |\mathcal{T}| \cdot \bigl( (1+o(1))n\bigr)^{m-1} \geq (1+o(1))^m m^{m-2} n^{m-1}.
  \]
  The claim follows then by summing over all choices of $v \in V(\P{n})$, and dividing by $(m-1)!$ to account for double counting.
\end{proof}

\section{The percolation threshold}\label{s:perc}

The goal of this section is to prove Theorem~\ref{t:percolationthreshold}.
We start by showing that, in both the subcritial and the supercritical regimes, it follows from Lemma~\ref{l:permtreecount} that there are tree components of order $\Omega(n \log n)$.

\begin{lemma}\label{l:treecomp}
  For every $c > 0$ there exists $\alpha > 0$ such that the following holds.
  For $p=\frac{c}{n}$, whp $\P{n}_p$ contains a tree component of order $\alpha n\log n$.
\end{lemma}
\begin{proof}
  The proof follows by a standard second moment calculation.
  Indeed, letting $X$ be the number of tree components of order $m = \alpha n\log n =o \bigl(\bigl(\frac{n}{\log n}\bigr)^2\bigr)$, by Lemma~\ref{l:permtreecount} for any $\alpha \in (0,1)$
  \begin{align*}
    \mathbb{E}(X) &\geq t_m(\P{n})p^{m-1}(1-p)^{nm}\\
                  &\geq  (n+1)! \,(ec)^{m-1}\exp\bigl( -cm + o(m)\bigr) \\
                  &=\exp \Bigl( n \log n\bigl( 1 +  \alpha(1 + \log c - c) + o(1)\bigr)\Bigr).
  \end{align*}
  In particular, we can choose $\alpha >0$ sufficiently small such that $\alpha(\log c - c) > - \frac{1}{4}$ and so
  \[
    \mathbb{E}(X) \geq \exp \biggl( \Bigl(\frac{3}{4} + \alpha\Bigr) n \log n\biggr).
  \]

  On the other hand, if we consider $X$ as a sum of indicator random variables, the only pairs $(T_1,T_2)$ of trees with $T_1 \neq T_2$ for which the covariance can be positive are those with $V(T_1)$ and $V(T_2)$ disjoint, but connected by an edge in $\P{n}$.
  In those cases, $V(T_1) \cup V(T_2)$ is a connected set of order $2m$.
  By Lemma~\ref{l:treecount}, there are at most
  \[
    (n+1)!\, (en)^{2m-1} \leq  \exp \bigl( (1+2\alpha)n \log n\bigr) = o\bigl(\mathbb{E}(X)^2\bigr),
  \]
  such pairs, each contributing at most $1$ to the variance of $X$.
  The pairs of the form $(T,T)$ contribute at most $\mathbb{E}(X)$ in total, and the remaining pairs of trees contribute at most $0$.
  Therefore, $\operatorname{Var}(X) = o\bigl(\mathbb{E}(X)^2\bigr)$ and the result follows by Chebyshev's inequality.
\end{proof}

We now focus on the supercritical case of Theorem~\ref{t:percolationthreshold}.
It remains to show the existence of a unique component of order $\omega(n \log n)$ in the supercritical regime, and to estimate its order.
To this end, let us fix $c > 1$, $\eps = c-1$ and $p = \frac{c}{n} = \frac{1+\eps}{n}$.
We will argue using a multi-round exposure, or sprinkling, argument.
Let $p_1 = \frac{1 + \eps/2}{n}$ and $p_2$ be such that $(1-p_1)(1-p_2) = 1-p$.
It is easy to check that $p_2 \geq p - p_1 = \frac{\eps}{2n}$.
Let us define
\[
  G_1 = \P{n}_{p_1}\quad\text{and}\quad G_2 = G_1 \cup \P{n}_{p_2},
\]
so that $G_2 \sim \P{n}_p$.
Recalling that $r = n^{14}$, let us further define $W_i = V_{\geq r}(G_i)$ and note that $W_1 \subseteq W_2$.

To prove Theorem~\ref{t:percolationthreshold} in the supercritical regime, we will show that $W_2$ has the correct order and is connected, and that all other components of $G_2$ have size $O(n \log n)$.
The three key steps can be summarized in the lemmas below, which will be proved later in this section.
The first one follows directly from Lemma~\ref{l:PFS} and Azuma--Hoeffding (Lemma~\ref{l:azuma}).

\begin{lemma}\label{l:V_0}
  With high probability, $|W_2| = (\gamma(c)+o(1)) (n+1)!$.
\end{lemma}

Given a graph $G$, a partition $\{A_1,A_2\}$ of a set $X \subseteq V(G)$ with $A_1 \neq \emptyset \neq A_2$ is \emph{$G$-component respecting} if no component of $G$ meets both $A_1$ and $A_2$.
The second lemma uses expansion properties of the permutahedron (Proposition~\ref{p:edge-iso-general}) and Lemma~\ref{l:dense1}.
\begin{lemma}\label{l:partitions}
  Whp for every $G_1$-component respecting partition $\{A_1,A_2\}$ of $W_1$ there is a path in $G_2$ between $A_1$ and $A_2$.
\end{lemma}

The third lemma is an application of Lemma~\ref{l:dense2}.
\begin{lemma}\label{l:unique-large-component}
  There exists a constant $C=C(\eps)$ such that whp every component of $G_2$ of size at least $Cn \log n$ meets $W_1$.
\end{lemma}

We now show how the above lemmas easily imply the proof of Theorem~\ref{t:percolationthreshold}.

\begin{proof}[Proof of Theorem~\ref{t:percolationthreshold} assuming Lemmas~\ref{l:V_0} to~\ref{l:unique-large-component}]
  In the subcritical regime, the upper bound on the order of the largest component follows directly from known results on percolation on regular graphs, see~\cite[Proposition 1]{NP10} or~\cite[Theorem 2]{DEKK22}.

  For the supercritical case, let $\alpha > 0$ and $C > 0$ be given by Lemmas~\ref{l:treecomp} and~\ref{l:unique-large-component}, respectively.

  By Lemma~\ref{l:partitions}, whp every vertex in $W_1$ lies in the same connected component in $G_2$.
  However, by Lemma~\ref{l:unique-large-component}, whp every component of $G_2$ in $W_2$ meets $W_1$.
  It follows that $G_2[W_2]$ is connected and hence $L_1 := G_2[W_2]$ is the (unique) largest component of $G_2$, since all other components of $G_2$ have size smaller than $r$.
  By Lemma~\ref{l:V_0}, whp $|L_1| = |W_2|= (\gamma(c)+o(1))(n+1)!$.

  Furthermore, by Lemma~\ref{l:treecomp} whp $G_2$ contains a component of order $\alpha n \log n$, whereas by Lemma~\ref{l:unique-large-component} whp there are no components of $G_2$ of order at least $Cn\log n$ which do not meet $W_1 \subseteq W_2 = V(L_1)$, and hence the second-largest component of $G_2$ has order $\Theta(n \log n)$.
\end{proof}

We return to the proof of the lemmas.
We start by showing that whp $|W_2|$ has the correct size.

\begin{proof}[Proof of Lemma~\ref{l:V_0}]
  Since $G_2 \sim \P{n}_p$, we may apply Lemma~\ref{l:PFS} and obtain that, for $v \in V(\P{n})$,
  \[
    \mathbb{P}(v\in W_2)\geq \gamma(c)+o(1).
  \]
  Conversely, since $\P{n}$ is $n$-regular, for every $v\in V(\P{n})$ we can couple a BFS exploration process in $G_2$ from above with a Galton--Watson tree with offspring distribution $\Bin(n,p)$.
  In particular, Lemma~\ref{l:branching}~\eqref{i:branching:infinite} and~\eqref{i:branching:survival} implies that for every $v\in V(G)$, $\mathbb{P}(v\in W_2)\leq \gamma(c)+o(1)$.
  Hence, $\mathbb{E}[|W_2|] = (\gamma(c)+o(1)) (n+1)!$.

  It remains to show that $|W_2|$ is well concentrated about its mean.
  To this end, let us consider the family of random variables $\{X_e : e \in E(\P{n})\}$, where $X_e$ is the indicator of the event that $e$ lies in $G_2$.
  Clearly $|W_2|$ is a function of this set of random variables, and it is clear that changing a single one of the $X_e$ can change $|W_2|$ by at most $2r$.
  Hence, by the Azuma--Hoeffding inequality (Lemma~\ref{l:azuma}),
  \[
    \mathbb{P}\Bigl(\bigl||W_2|- \mathbb{E}\bigl[|W_2|\bigr]\bigr| \ge ((n+1)!)^{2/3}\Bigr) \le 2\exp\biggl(-\frac{((n+1)!)^{4/3}}{8|E(\P{n})|r^2}\biggr).
  \]
  Since $|E(\P{n})| = n(n+1)!/2$ and $r = n^{14}$, the right-hand side tends to zero, proving that $|W_2|$ is concentrated around its expected value and finishing the proof of the lemma.
\end{proof}

We now show that whp the sprinkling step merges all large components of $G_1$.

\begin{proof}[Proof of Lemma~\ref{l:partitions}]
  We first expose $G_1$.
  Note that by Lemma~\ref{l:dense1} whp every vertex in $\P{n}$ is within distance two of $W_1$.
  We will assume in what follows that this holds deterministically.

  Let us fix a $G_1$-component respecting partition $\{A_1,A_2\}$ of $W_1$, where without loss of generality $a:=|A_1|\leq |A_2|$.
  Note that, since all components of $G_1[W_1]$ have size at least $r$, it follows that $a \geq r$.
  By our assumption, we can extend $\{A_1,A_2\}$ to a partition $\{A'_1,A'_2\}$ of $V(\P{n})$ such that every vertex in $A'_i$ is within distance two of $A_i$ for $i=1,2$, where we note that $|A'_i|\geq |A_i|\geq r$ for $i=1,2$.

  By Proposition~\ref{p:edge-iso-general} there are at least
  \[
    \Omega\biggl(\frac{\min \{|A'_1|,|A'_2|\}}{n^2}\biggr) = \Omega\Bigl(\frac{a}{n^2}\Bigr)
  \]
  edges in $\P{n}$ between $A'_1$ and $A'_2$.
  By construction, we can extend these edges to a (not necessarily disjoint) family of $A_1-A_2$ paths of length at most five.
  Then, since $\Delta(\P{n})= n$, each path shares an edge with at most $5n^4$ other paths, and so we can naively thin this family out to an edge-disjoint family of size $\Omega\bigl(a/n^6\bigr)$.
  In particular, since $p^5_2=\Omega\bigl(n^{-5}\bigr)$, the probability that none of these paths are contained in $\P{n}_{p_2}$ is at most
  \[
    \bigl(1-p^5_2\bigr)^{\Omega(a/n^6)} = \exp\Bigl(-\Omega\Bigl(\frac{a}{n^{11}}\Bigr)\Bigr).
  \]

  Since $W_1$ contains at most $(n+1)!$ components, and $A_1$ is the union of at most $a/r$ of them, there are at most
  \[
    \bigl((n+1)!\bigr)^{a/r} \leq \exp\biggl( \frac{an^2}{r}\biggr)
  \]
  $G_1$-component respecting partitions $\{A_1,A_2\}$ of $W_1$ with $|A_1|=a$.
  Hence, since $r=n^{14}$, by a union bound the probability that the statement does not hold is at most
  \[
    \sum_{a=r}^{|W_2|/2} \exp\biggl( \frac{an^2}{r}\biggr) \exp\Bigl(-\Omega\Bigl(\frac{a}{n^{11}}\Bigr)\Bigr) \leq \sum_{a \geq r} \exp\Bigl(-\Omega\Bigl(\frac{a}{n^{11}}\Bigr)\Bigr) = o(1).\qedhere
  \]
\end{proof}

Finally, we show that whp every large component of $G_2$ meets $W_1$.

\begin{proof}[Proof of Lemma~\ref{l:unique-large-component}]
  We first expose $G_1$, choose $\alpha=\alpha(\eps)$ sufficiently small and choose $C=C(\eps,\alpha)$ sufficiently large.
  Let $W_1^{\mathsf{c}} = V(\P{n}) \setminus W_1$ and
  \[
    X = \bigl\{\, v \in W_1^{\mathsf{c}} : \bigl|W_1 \cap N_{\P{n}}(v) \bigr| > \alpha n \,\bigr\}.
  \]

  By Lemma~\ref{l:dense2} whp every subset $M \subseteq W_1^{\mathsf{c}}$ of size $C n \log n$ which is connected in $\P{n}$ is such that $|M \setminus X| \leq \alpha n$.
  In particular, if $M \subseteq W_1^{\mathsf{c}}$ is a connected subset of $\P{n}$ of order at least $Cn \log n$ then
  \begin{equation}\label{e:edgesbetween}
    \bigl|E_{\P{n}}(M,W_1)\bigr| \geq (Cn\log n - \alpha n) \cdot \alpha n \geq  \frac{C}{2} \alpha n^2 \log n.
  \end{equation}
  We assume in what follows that this holds deterministically.

  We now expose $\P{n}_{p_2}$ on $W_1^{\mathsf{c}}$.
  We note that any component of $G_2$ which avoids $W_1$ is then a component $M$ of $G_2[W_1^{\mathsf{c}}]$ which is not adjacent to $W_1$ in $G_2$, and we have yet to expose the edges between $M$ and $W_1$ in $\P{n}_{p_2}$.

  In particular, for each component $M$ of $G_2[W_1^{\mathsf{c}}]$ of size at least $Cn \log n$, since $M$ spans a connected subset of $\P{n}$, by~\eqref{e:edgesbetween} the probability that there are no edges from $M$ to $W_1$ in $G_2$ is at most
  \[
    (1-p_2)^{(C\alpha n^2 \log n)/2} \leq \exp \Bigl( - \frac{C\alpha\eps}{4} \cdot n \log n \Bigr).
  \]
  However, there are at most $(n+1)!$ components of $G_2[W_1^{\mathsf{c}}]$, and so by a union bound the probability that the conclusion of the lemma fails to hold is at most
  \[
    (n+1)!\, \exp \Bigl( - \frac{C\alpha\eps}{4} \cdot n \log n \Bigr) = o(1),
  \]
  as long as $C=C(\eps,\alpha)$ is sufficiently large.
\end{proof}

Therefore, Theorem~\ref{t:percolationthreshold} is now unconditionally proven.

\begin{remark}\label{r:unique-large-component}
  We note that in fact, since the proofs of Lemmas~\ref{l:partitions} and~\ref{l:unique-large-component} only require a lower bound on $p_2$, the same argument will show that for any (potentially growing) $c\geq2$ and $p = \frac{c}{n}$ whp $\P{n}_p$ has a unique component of order at least $C(1) n \log n$, where $C(1)$ is as in Lemma~\ref{l:unique-large-component}.
\end{remark}

\section{The connectivity threshold}\label{s:conn}
We move on to Theorem~\ref{t:connectivitythreshold}, determining the connectivity threshold in $\P{n}_p$.
Let $\lambda = \lambda(n,p) = (n+1)!\,(1-p)^n$, and note that Stirling's approximation implies that $\lambda \to 0$ if $(1-p)n < c$ for any constant $c < e$, and $\lambda \to \infty$ if $(1-p)n > C$ for any constant $C > e$.
We begin by showing that, already well before the threshold in Theorem~\ref{t:connectivitythreshold}, all components in $\P{n}_p$ are either isolated vertices, or exponentially large.
\begin{lemma}\label{l:no-medium-components}
  Let $(1-p)n \leq 3$.
  Then whp $\P{n}_p$ has no components of order $2 \leq k \leq 2^{n/4}$.
\end{lemma}
\begin{proof}
  If $K$ is a component of order $k \leq 2^{n/4}$, then $K$ contains a tree $T$ of order $k$, all of whose edges are present in $\P{n}_p$ and such that each edge in $\partial (V(K))$ is not present in $\P{n}_p$.
  Note that, by Proposition~\ref{p:edge-iso-general}
  \[
    |\partial (V(K))| \geq k (n- \log_2 k) \geq \frac{3kn}{4}.
  \]
  Hence, by Lemma~\ref{l:treecount}, the probability that there is a component of order $k$ with $2 \leq k \leq 2^{n/4}$ is at most
  \[
    \sum_{k=2}^{2^{n/4}} (n+1)!\,(en)^{k-1} p^{k-1} (1-p)^{3kn/4} \leq \sum_{k=2}^{2^{n/4}} (n+1)!\, (en)^{k}  \Bigl(\frac{3}{n} \Bigr)^{3kn/4}.
  \]
  It is easy to check that the summand on the right-hand side is decreasing in $k$.
  Therefore, since $n! \leq n^n$, the right-hand side is at most \[ \exp\Bigl( n \log n - \frac{6n}{4} \log n + O(n) \Bigr), \] which tends to zero.
\end{proof}

\begin{proof}[Proof of Theorem~\ref{t:connectivitythreshold}]
  We note first that, since the property of being connected is monotone, we may assume without loss of generality that $2 \leq (1-p)n \leq 3$.
  Hence, by Lemma~\ref{l:no-medium-components} and Remark~\ref{r:unique-large-component} whp there is a unique component of order at least $2$.
  In particular,
  \[
    \mathbb{P}(\P{n}_p \text{ is connected}) = \mathbb{P}(\P{n}_p \text{ contains no isolated vertices}) + o(1).
  \]

  If we let $X$ be the number of isolated vertices in $\P{n}_p$, it is a simple exercise to estimate the moments of $X$.
  Indeed, given any $r \in \mathbb{N}$ and a subset $S \subseteq V(\P{n})$ of size $r$, by Proposition~\ref{p:edge-iso-general} there are at most $rn$ and at least $r(n-\log_2 r)$ edges meeting $S$.
  Furthermore, there are at most $\bigl((n+1)!\bigr)_{r-1}\cdot rn$ sets of size $r$ meeting fewer than $rn$ edges, where $x_r = x(x-1)\cdots(x-r+1)$ is the $r$th falling factorial.
  It follows that
  \[
    \bigl((n+1)!\bigr)_r(1-p)^{rn} \leq \mathbb{E}[X_r] \leq \bigl((n+1)!\bigr)_r(1-p)^{rn} + \bigl((n+1)!\bigr)_{r-1}rn (1-p)^{r(n-\log_2 r)}.
  \]
  Recalling that $\lambda = (n+1)!\,(1-p)^n$ and $1 -p \geq \frac{2}{n}$, elementary estimates lead to
  \[
    \lambda^r\biggl( 1- \frac{r^2}{(n+1)!}\biggr) \leq \mathbb{E}[X_r] \leq \lambda^r\biggl(1 +\frac{rn}{(n+1)!\,(1-p)^{r \log_2 r}}\biggr),
  \]
  and both sides are $(1+o(1))\lambda^r$.
  In particular, if $\lambda \to c \in \mathbb{R}$, then by the method of moments (see for example~\cite{JLR00}) $X$ tends in distribution to a Poisson distribution with mean $c$ and so
  \[
    \mathbb{P}(\P{n} \text{ is connected}) = \mathbb{P}(X=0) + o(1) \xrightarrow{ n \to \infty } \mathbb{P}(\operatorname{Po}(c)=0) = e^{-c}.
  \]

  Otherwise, in the case $\lim_{n \to \infty} \lambda = \infty$ or $\lim_{n \to \infty} \lambda = 0$, the result follows by a simple first or second moment argument, respectively.
\end{proof}

It is relatively straightforward to use these ideas to give a corresponding \emph{hitting time} result for the random graph process on $\P{n}$.
That is, if we build a random subgraph of $\P{n}$ by adding one edge at a time, each time choosing uniformly at random from the remaining edges, then whp this subgraph will become connected precisely when the last isolated vertex disappears.

\section{Discussion}\label{s:discuss}
As we saw in Sections~\ref{s:prelim},~\ref{s:projection} and~\ref{s:PFS}, much of the work in this paper generalises to the case of Cayley graphs of Coxeter groups or zonotopes.
However, a key missing part is some control over the large scale isoperimetric properties of these graphs, as in Proposition~\ref{p:edge-iso-general}.
In the case of the permutahedron, spectral methods give a close to optimal bound on the Cheeger constant, up to a small polynomial factor in the dimension/regularity.
This is then a key ingredient to the proofs of Theorems~\ref{t:percolationthreshold} and~\ref{t:connectivitythreshold}, allowing us to merge the large clusters we grow in a sprinkling step.
However, a much weaker bound on the Cheeger constant, for example anything sub-exponential, would be sufficient in this argument.

It is therefore interesting to ask if we can give any general bounds on the Cheeger constant of these classes of graphs.
In the case of zonotopes, we make a rather bold conjecture.
\begin{conjecture}
  Let $Z = \sum_{i=1}^k [-\bm{v_i},\bm{v_i}]$ be a zonotope and let $G$ be its $1$-skeleton.
  Then $i(G)$ is at worst inverse polynomial in $k$.
\end{conjecture}
Here, it would seem that the `worst' examples should come from cycles of length $2k$, which have Cheeger constant around $\frac{2}{k}$.
There is a similar, well-known conjecture of Mihail and Vazirani~\cite{M92} on the expansion of $0/1$ polytopes, where $Q^d$ is conjectured to have the worst expansion.
The specific case of matroid basis polytopes was solved in a breakthrough result very recently~\cite{ALGV24}.

More generally, it would be interesting to know how the phase transition at the percolation threshold develops in other classes of polytopes --- how the critical probability relates to the degree distribution and the largest cluster sizes in the sub- and super-critical regimes.
If we consider the Cartesian powers of some fixed low-dimensional polytope with an irregular $1$-skeleton, the negative answer to Question~\ref{q:generalised-many-stars} shows that we do not always have this dichotomy between logarithmic and linear sized clusters.
However, it would be very interesting to know if the quantitative behaviour in Theorems~\ref{t:ER},~\ref{t:cube} and~\ref{t:percolationthreshold} was perhaps universal to `compatible' sequences of \emph{simple} polytopes.

In~\cite{EKK22,DEKK23+}, similar methods to those in Section~\ref{s:perc} are used to not only show the existence of a giant component in the supercritical percolated graphs, but to demonstrate that this giant component likely has good expansion properties, from which bounds on various structural parameters of the giant component such as its diameter, circumference and mixing time can be derived.
\begin{question}\label{q:structure}
  What can we say about the structural properties of the giant component $L$ of a supercritical percolated $n$-dimensional permutahedron?
  \begin{enumerate}
  \item\label{i:one} Does $L$ whp contain a path of length $\Theta((n+1)!)$?
  \item What is the likely diameter of $L$?
  \item\label{i:three} What is the likely mixing time of a lazy random walk on $L$?
  \end{enumerate}
\end{question}
Using the methods of~\cite{EKK22,DEKK23+} together with the tools developed in this paper, it should be relatively straightforward to obtain \emph{some} lower bound on the expansion of the giant component in $\P{n}$, in particular one which is inverse polynomial in $n$, which would give answers to Question~\ref{q:structure}~\eqref{i:one}--\eqref{i:three} which are tight up to some power of $n$.
However, it seems unlikely that one could obtain an optimal bound (even up to polylogarithmic factors) as in~\cite{DEKK23+} via these methods without first understanding better the isoperimetric properties of $\P{n}$, in particular for large sets.
For this reason, we have not made much attempt to optimise the quantitative aspects of the arguments in Section~\ref{s:perc}.

Furthermore, for the diameter and mixing time, it is not obvious what natural lower bounds there are for these quantities.
In the case of $G(n,p)$ and $Q^n_p$, similar arguments as in Lemma~\ref{l:treecomp} show the likely existence of a bare path of length $\Omega(\log n)$ and $\Omega(n)$ respectively, leading to natural lower bounds on the diameter and mixing time in these cases.
However, the methods of Lemma~\ref{l:treecomp} only apply to trees with small depth and maximum degree, and so in particular are not effective for counting paths in $\P{n}$.
For this reason, it would be interesting to know what the length of the longest bare path in $\P{n}_p$ is, and to this end to effectively estimate the number of paths of length $\Theta(n \log n)$ in $\P{n}$.
From the other side, the diameter of $\P{n}$ is clearly $\binom{n}{2}$ and the mixing time of the lazy random walk on $\P{n}$ is known to be $\Theta(n^2 \log n)$~\cite{W04,L16}.
Whilst neither parameter is increasing under taking subgraphs, analogues to the case of $G(n,p)$ and $Q^n_p$ suggest that in the supercritical regime we should expect the giant component to have a diameter and mixing time which is slightly larger (by a logarithmic factor) than that of the host graph.

As mentioned above, whilst we can determine quite precisely the expansion of small sets in $\P{n}$, the isoperimetric properties of $\P{n}$ in general seem to not be very well-understood.
We discussed the edge-isoperimetric problem in Section~\ref{s:iso}, but the analogous \emph{vertex-isoperimetric problem} in $\P{n}$ is also very interesting.
A reasonably natural conjecture here would be that the vertex-boundary is minimised by Hamming balls in Kendall's $\tau$-metric (the distance metric on $S_{n+1}$ induced by $\P{n}$).

Moving on to the connectivity threshold, in both $G(n,p)$ and $Q^n_p$, it has been shown that the connectivity threshold is asymptotically the same as the threshold for containing a perfect matching and Hamiltonian cycle.
In the case of the hypercube, both the connectivity threshold and the threshold for the existence of a perfect matching can be shown using quite delicate but elementary first moment method arguments, which rely on strong and explicit isoperimetric inequalities for the hypercube, in particular Theorem~\ref{t:Harper}.
On the other hand, the threshold for containing a Hamiltonian cycle in this model was only recently resolved, in groundbreaking work by Condon, Espuny D{\'\i}az, Gir{\~a}o, K{\"u}hn and Osthus~\cite{CDGKO21}.

However, in the case of the permutahedron, even the existence of a Hamiltonian cycle in the host graph is non-trivial (see, e.g.,~\cite[Section 7.2.1.2]{K11}).
\begin{question}
  What is the threshold for the existence of a perfect matching or Hamiltonian cycle in $\P{n}_p$?
  Do they coincide with the connectivity threshold?
\end{question}

Finally, viewing the permutahedron as a $\mathcal{V}$-polytope rather than as a graph, there is perhaps another natural notion of a random substructure of $P(n)$ given by taking a random subset of the permutations, given by including each permutation independently and with probability $p$, and considering the convex hull of these points, which we denote by $P(n,p)$.
From the other direction, considering the permutahedron as an $\mathcal{H}$-polytope, we could equally consider the constraints determining the permutahedron as an intersection of hyperplanes and choose a random subset of these constraints.
\begin{question}
  Let $p \in (0,1)$.
  \begin{enumerate}
  \item What is the likely volume of $P(n,p)$?
  \item What is the likely number of lattice points inside $P(n,p)$?
  \item What is the likely number of facets of $P(n,p)$?
  \item What is the likely number of edges of $P(n,p)$?
  \end{enumerate}
\end{question}
In the case of $P(n) = P(n,1)$ these parameters have simple combinatorial interpretations, and it would be interesting to see if there are combinatorial, or stochastically combinatorial, interpretations of the corresponding parameters in $P(n,p)$.
The model is also closely related to the model of random $0/1$ polytopes considered in~\cite{LR22}, where they showed that a weaker form of Mihail and Vazirani's conjecture holds for almost all $0/1$ polytopes.
This result was later strengthened by Ferber, Krivelevich, Sales and Samotij~\cite{FKSS26}.
It would be interesting to know if the typical expansion in $P(n,p)$ is always smaller than that of $P(n)$, or even just if it can be bounded as an inverse polynomial in $n$.

\subsection*{Acknowledgements}
The authors would like to thank Sahar Diskin for useful conversations regarding the Projection-First Search algorithm and Cesar Ceballos for suggesting the model. The authors would also like to thank the anonymous referees for their diligent work and detailed suggestions, which greatly improved the manuscript.

This research was funded in whole or in part by the Austrian Science Fund (FWF) [10.55776/P36131, 10.55776/P33278].
For open access purposes, the author has applied a CC BY public copyright license to any author accepted manuscript version arising from this submission. M.C. was supported by CNPq (420838/2025-2) and FAPESP (2023/03167-5, 2024/13859-4). This study was financed in part by the Coordenação de Aperfeiçoamento de Pessoal de Nível Superior -- Brasil (CAPES) -- Finance Code 001.

\bibliographystyle{abbrv}
\bibliography{perm}

\appendix
\section{Modified projection-first search and its applications}\label{a:modified-PFS}
We will use the following lemma about the expectation of a truncated binomial random variable.

\begin{lemma}\label{l:restricted}
  Let $\beta > 0$, let $m \in \mathbb{N}$, let $p= \frac{1 + \beta}{m}$, let $m' \geq \bigl(1 - \min\bigl\{\frac{\beta}{2}, \frac{1}{18}\bigr\}\bigr)m$, and let $K \geq \max\{2emp, \log_2 (\beta^{-2})\}$.
  If $X \sim \Bin(m',p)$ and $Y = \min\{X , K\}$, then
  \[
    \mathbb{E}[Y] \geq 1+\frac{\beta}{4}.
  \]
\end{lemma}

\begin{proof}
  If $m' > m$, then $X$ stochastically dominates $X' \sim \Bin(m, p)$.
  Therefore, it suffices to prove the lemma when $m' \leq m$, since replacing $X$ by $X'$ can only decrease $\mathbb{E}[Y]$.

  By Lemma~\ref{l:chernoff}~\eqref{i:chernoff2} and the fact that $K \geq 2em'p$, we have
  \[
    \mathbb{E}[X - Y] = \sum_{t=K+1}^n \mathbb{P}(X \geq t) \leq \sum_{t=K+1}^n \Bigl(\frac{em'p}{t}\Bigr)^{t} \leq \sum_{t=K+1}^\infty \Bigl(\frac{1}{2}\Bigr)^t = 2^{-K}.
  \]
  Let $1/9 \leq \beta_0 \leq 1/8$ be the solution to the equation $2e(1+x) = \log_2(x^{-2})$.
  If $0 < \beta \leq \beta_0$, we have $K \geq \log_2(\beta^{-2})$, and therefore $\mathbb{E}[Y] \geq \mathbb{E}[X] - \beta^2 \geq \bigl(1- \frac{\beta}{2}\bigr)(1 + \beta) - \beta^2 \geq 1 + \frac{\beta}{4}$.
  Otherwise, $\beta > 1/9$ implies $m' \geq 17m/18$, and therefore $\mathbb{E}[Y] \geq \mathbb{E}[X] - 2^{-2emp} \geq \frac{17(1+\beta)}{18} - 2^{-2e(1+\beta)} \geq 1+\frac{\beta}{4}$.
\end{proof}

We start by proving Lemma~\ref{l:PFSprecise}, as promised.

\begin{proof}[Proof of Lemma~\ref{l:PFSprecise}]
  The algorithm PFS$'$ is as follows: As in (the proof of) Lemma~\ref{l:PFS}, we first run the PFS exploration process starting at $v$ for $\tau_1 := \log \log m$ many rounds, conditioning at the start of each round $t\leq \tau_1$ on the event
  \[
    \mathcal{E}(t) = \text{``$w(x) \leq \log^2 m$ for each $x \in W(t)$''}.
  \]
  As in Lemma~\ref{l:PFS}, we can couple this initial exploration process from below with a Galton--Watson tree with child distribution $\Bin(m-\log^2 m,p)$.
  In particular, since $(1+\beta+o_m(1))^{\tau_1} \geq \log \log m$, by Lemma~\ref{l:branching} we have that
  \begin{equation}\label{e:A(1)toosmall}
    \mathbb{P}\bigl( |A(\tau_1+1)| <  \log \log m \bigm| \mathcal{E}(\tau_1)\bigr) \leq 1 - \gamma(1+\beta) + o_m(1).
  \end{equation}
  As in Lemma~\ref{l:PFS}, we have $\mathbb{P}\bigl(\mathcal{E}(\tau_1+1)^c\bigr) = \sum_{t=1}^{\tau_1} \mathbb{P}\bigl(\mathcal{E}(t) \cap \mathcal{E}(t+1)^c \bigr) = o_m(1)$, and thus any event $\mathcal{F}$ satisfies \begin{equation}\label{e:PstarP-asymp}
    \mathbb{P}\bigl(\mathcal{F} \bigm| \mathcal{E}(\tau_1)\bigr) = \mathbb{P}\bigl(\mathcal{F}\bigr) + o_m(1)\qquad\text{and}\qquad\mathbb{P}\bigl(\mathcal{F} \bigm| \mathcal{E}(\tau_1+1)\bigr) = \mathbb{P}\bigl(\mathcal{F}\bigr) + o_m(1).
  \end{equation}
  In what follows we condition on $\mathcal{E}(\tau_1+1)$, writing $\mathbb{P}^*$ for the probability in this conditional space.

  We now continue with a slightly modified PFS process, starting with frontier $A(\tau_1+1)$, except when we expose the neighbourhood $N_t(x)$ of a vertex $x$ we stop after we have found
  \[
    K := \max\{2emp, \log_2 (\beta^{-2})\}
  \]
  neighbours.
  To be more precise, in each round $t$, for each $x \in A(t)$ we expose the edges incident to $x$ in $H(x)$ one-by-one, stopping once $K$ neighbours have been discovered (or all edges have been exposed).

  Let us suppose we run this latter step for a further $\tau_2:=\min\{\frac{\beta}{3}, \frac{1}{20}\}\frac{m}{K}$ rounds.
  Since fact~\eqref{i:dimension} continues to hold in this modified process, if we condition on $\mathcal{E}(\tau_1+1)$, then deterministically for each $x \in A(\tau_1+\tau_2+1)$, we have
  \[
    w(x) \leq \log^2 m + K\tau_2 \leq \log^2 m + \min\Bigl\{\frac{\beta}{3}, \frac{1}{20}\Bigr\}m \leq \min\Bigl\{\frac{\beta}{2}, \frac{1}{18}\Bigr\} m.
  \]
  Hence, under this conditioning, during this second round, whenever we exposed the neighbourhood of a vertex $y$, the size of the neighbourhood $N_t(y)$ we discovered stochastically dominates a random variable distributed as $\min \bigl\{ \Bin\bigl(\bigl(1 - \min\bigl\{\frac{\beta}{2}, \frac{1}{18}\bigr\}\bigr)m, p\bigr),K\bigr\}$, which by Lemma~\ref{l:restricted} has expectation at least $1+\beta/4$.
  That is, for any $\tau_1 < t \leq \tau_1+\tau_2$, $|A(t+1)|$ stochastically dominates the sum of $|A(t)|$ independent variables, each with mean $1+\beta/4$ and bounded by $K$.

  Letting $\alpha = 1 + \beta/8$ and $L = \log \log m$, we claim that the above implies that $|A(\tau_1 + \tau_2 + 1)| \geq \alpha^{\tau_2} L$ holds with probability at least $\gamma(1+\beta) + o_m(1)$.
  Indeed, if $|A(\tau_1 + \tau_2 + 1)| < \alpha^{\tau_2} L$, there must exist a smallest $0 \leq t \leq \tau_2$ for which $|A(\tau_1 + t + 1)| < \alpha^t L$.
  Applying the union bound, we obtain
  \begin{align}
    &\mathbb{P}^*\bigl(|A(\tau_1+\tau_2+1)| < \alpha^{\tau_2} L \bigr)\nonumber\\
    &\leq \mathbb{P}^*\Bigl( |A(\tau_1+1)| < L\Bigr) + \sum_{t=1}^{\tau_2} \mathbb{P}^*\Bigl(|A(\tau_1+t+1)| < \alpha^{t}L \;\cap\; |A(\tau_1+t)| \geq \alpha^{t-1} L \Bigr)\label{e:union-bound-modified-pfs}.
  \end{align}
  Observe now that, for $c = \beta^2/(128 K^2)$, the Azuma--Hoeffding inequality (Lemma~\ref{l:azuma}) implies that for any $k \geq \alpha^{t-1} L$ such that $\mathbb{P}^*\bigl(|A(\tau_1+t)| = k\bigr) > 0$, we have
  \[
  \mathbb{P}^*\Bigl(|A(\tau_1+t+1)| < \alpha^t L \, \Bigm| \, |A(\tau_1+t)| = k \Bigr) \leq 2\exp(-ck).
  \]
  Therefore, by the law of total probability, the $t$th term of of the summation in~\eqref{e:union-bound-modified-pfs} is at most $2\exp\bigl(-c\alpha^{t-1}L\bigr) = o_m(1)$, and so the right-hand side of~\eqref{e:union-bound-modified-pfs} is at most
  \begin{align*}
    \mathbb{P}^*\bigl( |A(\tau_1+1)| < L\bigr) + \sum_{t=1}^{\tau_2} 2 \exp\bigl( - c\alpha^{t-1} L\bigr) \leq 1 -  \gamma(1+\beta) + o_m(1),
  \end{align*}
  where the inequality follows from~\eqref{e:A(1)toosmall} and~\eqref{e:PstarP-asymp}.
  Noting that $\alpha^{\tau_2} = \exp\bigl(\Theta(m)\bigr)$ and using~\eqref{e:PstarP-asymp} again, we obtain
  \[
    \mathbb{P}\bigl(|C_v| = e^{\Omega(m)}\bigr)\geq\mathbb{P}\bigl(|A(\tau_1+\tau_2+1)| \geq \alpha^{\tau_2} L \bigr) \geq \gamma(1+\beta) + o_m(1),
  \]
  finishing the proof.
\end{proof}

Let us compare the strength of this method to previous results on the phase transition in such models.

In the first proof of the existence of a giant component in $Q^n_p$ given by Ajtai, Koml\'{o}s and Szemer\'{e}di~\cite{AKS81}, they use a two-step argument to show that each vertex is contained in a cluster of size $\Omega(n^2)$ with some constant probability $c(\beta)$, although their methods can easily be used inductively to show that the same holds for clusters of size $n^k$ for arbitrary $k \in \mathbb{N}$, and with careful bookkeeping one can take $c(\beta)=\gamma(1+\beta) + o(1)$.
The later work of Bollob\'{a}s, Kohayakawa and {\L}uczak~\cite{BKL92} avoids having to estimate the probability that a vertex lies in a large cluster by instead demonstrating a gap in the order of the components and estimating the number of vertices contained in \emph{small} clusters.
However, their proof method relies on particular strong and explicit bounds on the isoperimetric profile of the hypercube, and so cannot be applied in other contexts.
In the work of Diskin, Erde, Kang and Krivelevich~\cite{DEKK22,DEKK23}, an inductive argument is used together with a projection lemma to show that each vertex is contained in a cluster of size $n^{r}$ for any $r = o\bigl( n^{1/3} \bigr)$ with probability $\gamma(1+\beta) + o(1)$, and a natural limit to their method seems to be $r = n^{1/2}$.

In comparison, using for example~\cite[Lemma 3.1]{DEKK22} as a projection lemma, the proof of Lemma~\ref{l:PFSprecise} shows that each vertex in $Q^n_p$ is contained in a cluster of size $2^{\Theta(n)}$ with probability at least $\gamma(1+\beta) + o(1)$.
This strengthening seems to have a few uses --- beyond improving some quantitative aspects of previous work (for example, it should be possible to improve the bound on the isoperimetric constant in~\cite[Theorem 8]{DEKK23} from $t^{-t^{1/4}}$ to something closer to $2^{-\Theta(t)}$), an immediate application is to percolation in irregular product graphs.

\subsection{Percolation in irregular product graphs}

In~\cite{DEKK23} it was observed that, whilst all regular product graphs undergo a quantitatively similar phase transition around the percolation threshold as $G(n,p)$, the same is not necessarily true in the irregular case.
In particular, if we consider the product of a number of copies of a star, a third regime in the phase transition appears, where there is still no giant component, but there are components of (almost linear) polynomial size.
The authors of~\cite{DEKK23} asked whether this holds in fact for \emph{any} irregular product graph.
\begin{question}[{\cite[Question 26]{DEKK23}}]\label{q:generalised-many-stars}
  For all $i\in [t]$, let $G^{(i)}$ be an irregular connected graph of order at most $C>0$.
  Let $G=\bigcart_{i=1}^t G^{(i)}$.
  Let $\eps>0$ be a small enough constant, and let $p= \frac{1-\eps}{d}$, where $d$ is the average degree of $G$.
  Does there exist a $c(\eps,C)>0$ such that whp the largest component in $G_p$ has order at least $|G|^c$?
\end{question}

Using the projection lemma from~\cite[Lemma 13]{DEKK23}, it is relatively straightforward to use the methods of Lemma~\ref{l:PFSprecise} to show that Question~\ref{q:generalised-many-stars} has a positive answer.
Let us sketch an argument here.
For background on properties of product graphs, see~\cite{DEKK22,DEKK23}.
Indeed, suppose $G=\bigcart_{i=1}^t G^{(i)}$ is as in Question~\ref{q:generalised-many-stars}.
Here, the projection lemma from~\cite[Lemma 13]{DEKK23} roughly tells us that given a small set $X$ of vertices we can find a disjoint collection of \emph{projections} whose codimension is at most $|X|-1$ which cover these vertices.
Here, a projection is an induced subgraph of $G$ on a set of vertices given by fixing some subset of the coordinates and letting the other coordinates vary, and the codimension of such a projection is the number of fixed coordinates.

Note that, since each $G^{(i)}$ is connected, $t \leq d_G(v) \leq Ct$. A particularly useful fact is that the degree function on $G$ is Lipschitz in the graph distance, in the sense that for any $v,w \in V(G)$, we have $|d_G(v) - d_G(w)| \leq (C-1) \dist_G(v,w)$.
Roughly, product graphs look `locally almost-regular'.
Furthermore, it is easy to see that if $H$ is a projection of $G$ of codimension at most $k$, then for any vertex $v \in H$, we have $d_G(v) - d_H(v) \geq (C-1)k$.

Suppose that we run the PFS$'$ exploration process in $G_p$ starting at a vertex $v$, where $p = \frac{1+\beta}{d_G(v)}$. Let $K = \max\{2e(1+\beta), \log_2 (\beta^{-2})\}$ be as in the proof of Lemma~\ref{l:PFSprecise}, and let $\beta' = \min\{\beta, \frac{1}{2}\}$.
Let $c = c(\beta, C)$ be sufficiently small, and suppose that we are considering a vertex $x \in W(ct)$ and its associated projection $H(x)$.
$H(x)$ will have codimension at most $(K + o(1))ct \leq \beta' d_G(x)/(6C)$ and $\dist_G(v,x) \leq \beta' d_G(x)/(6C)$ and so by the above, $d_{H(x)}(x) \geq \bigl(1 - \frac{\beta'}{3}\bigr)d_G(x)$.
In particular, if we run the process until depth $ct$, then throughout this period $p \cdot d_{H(x)}(x) \geq 1+ \frac{\beta}{2}$ and the tree should grow exponentially.
In this way, a similar argument as in the proof of Lemma~\ref{l:PFSprecise} will show the following.

\begin{lemma}\label{l:PFSirreg}
  Let $H=\bigcart_{i=1}^t H^{(i)}$ be a product graph whose factors have bounded size.
  Let $\beta > 0$, let $v \in V(H)$ and let $p \cdot d_H(v) \geq 1 + \beta$.
  Then
  \[
    \mathbb{P}\bigl(|C_v| = e^{\Omega(t)}\bigr)\geq \gamma(1+\beta) +o_t(1).
  \]
\end{lemma}

Let $d(H)$ denote the average degree of a graph $H$.
Since each $G^{(i)}$ is irregular and of bounded size, it is easy to see that, for each $i \in [t]$,
\[
  \Delta(G^{(i)}) \geq d(G^{(i)}) + \frac{1}{C} \geq \Bigl(1 + \frac{1}{C^2}\Bigr) d(G^{(i)}).
\]
One may check that $d(G) = \sum_{i=1}^t d(G^{(i)})$ and hence $\Delta(G) \geq \bigl(1+\frac{1}{C^2}\bigr)d(G)$.
In particular, we can choose $\eps>0$ sufficiently small such that $\frac{1-\eps}{d(G)} \cdot \Delta(G) > 1 + \eps/2$.

Then, since $G$ is locally almost-regular, we can find a set $X$ of size $\log^2 t$ such that $d_G(x) = (1+o_t(1))\Delta(G)$ for each $x \in X$, for example by perturbing a vertex of maximum degree in a small number of coordinates.
If we apply the projection lemma to this set $X$, we obtain a disjoint family $\{H(x) : x \in X\}$ of projections of codimension at most $\log^2 t$ and hence $d_{H(x)}(x) = (1+o_t(1))\Delta(G)$ for each $x \in X$.

Applying Lemma~\ref{l:PFSirreg} to each $x \in X$ in $H(x)$ we see that each $x$ has a positive probability to be contained in a component of order $\exp\bigl(\Omega(t)\bigr)$ in $H(x)_p$, where $p = \frac{1-\eps}{d(G)}$.
Since these events are independent for different $x$, the probability that none of them happen is $o_t(1)$.
Finally, since
\[
  |G| = \prod_{i=1}^t |G^{(i)}| = \exp\bigl(O(t)\bigr),
\]
an affirmative answer to Question~\ref{q:generalised-many-stars} follows.
In fact, all we require for the above argument is that $p \cdot \Delta(G) \geq 1+\eps$.
Note that, conversely, if $p \cdot \Delta(G) \leq 1 - \eps$, then~\cite[Theorem 2]{DEKK22} implies that whp the largest component in $G_p$ has order $O( \log |G| )$.

As an explicit example, this implies that the phase-transition in a high-dimensional grid $[3]^n$ is qualitatively different to the phase-transition on a high-dimensional torus $\mathbb{Z}_3^n$.
Indeed, in $\mathbb{Z}_3^n$ it is known (see, e.g.~\cite{DEKK22}) that there is a sharp threshold at $p=\frac{1}{2n}$ for the emergence of a giant component, and that below this threshold whp the largest component is of order $O(n)$, and so is logarithmic in the size of the base graph.
In contrast, together with the results of~\cite{DEKK23}, the above implies that in $G = [3]^n$ there are at least three distinct regimes:
\begin{itemize}
\item a \emph{subcritical} regime when $p < \frac{1}{\Delta(G)} = \frac{1}{2n}$ in which whp the largest component has logarithmic order;
\item a \emph{supercritical} regime when $p > \frac{1}{d(G)} = \frac{3}{4n}$ in which whp there is a unique component of linear order;
\item an \emph{intermediary} regime when $\frac{1}{2} < pn < \frac{3}{4}$ in which whp the largest component is sublinear, but still has polynomial order.
\end{itemize}
It would be very interesting to determine more precisely the size of the largest components in this intermediary regime, and we hope to study this problem in more detail in future work.

\section{A projection lemma for bipartite Kneser graphs}\label{a:kneser}

Recall that, given $n,k$ with $k < n/2$ the \emph{bipartite Kneser graph} $H(n,k)$ is a bipartite graph with partition classes $[n]^{(k)}$ and $[n]^{(n-k)}$, where a subset in the first partition class is adjacent to each of its supersets in the second partition class.
It will be convenient for us to consider a slightly more general class of graphs defined as follows: given $n,k_1,k_2$ with $k_1 < n-k_2$ we define $H(n,k_1,k_2)$ to be a bipartite graph with partition classes $[n]^{(k_1)}$ and $[n]^{(n-k_2)}$, where a subset in the first partition class is adjacent to each of its supersets in the second partition class.
Note that, with this definition, $H(n,k) = H(n,k,k)$.
It will be convenient to think of the vertex set of $H(n,k_1,k_2)$ as living in $\{0,1\}^n$ via the standard association of a set with its characteristic vector.

In the particular case $k_1+k_2=n-1$, the graph $H(n,k_1,k_2)$ is isomorphic to the induced subgraph of the hypercube $Q^n$ between two adjacent layers, i.e., the sets of vertices with Hamming weight $k_1$ and $n-k_2 = k_1+1$.
%Furthermore, whenever $k_1,k_2 =(1+o(1)) n/2$, this graph is \emph{approximately $n/2$-regular}.
Furthermore, when $n$ is odd and $k=\frac{n-1}{2}$, the bipartite Kneser graph $H(n,k)$ is known as the \emph{middle layers graph} and is denoted by $M_n$.

Given any two disjoint subsets $I,J \subseteq [n]$ we can consider the subgraph of $H(n,k_1,k_2)$ induced by the vertex set
\[
  V(I,J) = \{ x \in V(H(n,k_1,k_2)) : x_i = 1 \text{ for all } i \in I \text{ and } x_j = 0 \text{ for all } j \in J\}.
\]
It is easy to see that
\[
  H(n,k_1,k_2)\bigl[V(I,J)\bigr] \cong H\bigl(n - |I| - |J|, k_1 - |I|, k_2 - |J|\bigr).
\]

%Note that, when $k_1+k_2 = n-1$ and $|I| = |J|$ we have that $k_1 - |I| + k_2 - |J| = n- |I|-|J| -1$ and this induced subgraph is again isomorphic to an induced subgraph of a hypercube between two adjacent layers.
%Furthermore, if $k_1,k_2 =(1+o(1)) n/2$ and $|I|=|J| = o(n)$, then $H(n,k_1,k_2)[V(I,J)]$ is still approximately $n/2$-regular.

\begin{lemma}
  Let $n$ be an odd integer and let $X \subseteq M_n$ have order $|X| \leq n/4$.
  Then there exists a disjoint family of subgraphs $\{M(x) \colon x \in X\}$ of $M_n$ such that
  \begin{itemize}
  \item For each $x \in X$ there exists $i(x)$ such that $i(x) \leq 2(|X|-1)$ and $M(x) \cong M_{n-i(x)}$;
  \item $x \in V(M(x))$ for all $x \in X$.
  \end{itemize}
\end{lemma}

\begin{proof}
  For distinct $x,y \in X$, let $k(x,y)$ be an arbitrarily-chosen coordinate in which $x$ and $y$ differ and let $K(x) = \{k(x,y) : y \in X \setminus \{x\}\}$.
  It is clear that we can choose disjoint sets $I(x),J(x) \subseteq [n]$ such that the following hold:
  \begin{itemize}
  \item $K(x) \subseteq I(x) \cup J(x)$;
  \item $\max\{|I(x)|, |J(x)|\} \leq |X|-1$;
  \item $x_i =1$ for all $i \in I(x)$; and
  \item $x_j = 0$ for all $j \in J(x)$.
  \end{itemize}
  By enlarging $I(x)$ or $J(x)$ if necessary, we can also assume that $|I(x)| = |J(x)|$.
  We let
  \[
    M(x) = M_n\bigl[V(I(x),J(x))\bigr],
  \]
  noting that by construction $x \in V(M(x))$.
  Writing $k = \frac{n-1}{2}$ and $i(x) = |I(x)| + |J(x)|$, we have
  \[
    M(x) \cong H\bigl(n-i(x), k-|I(x)|, k-|J(x)|\bigr) = M_{n-i(x)},
  \]
  since $k-|I(x)| = k-|J(x)| = \frac{n-i(x)-1}{2}$ and $i(x) \leq 2(|X|-1)$.
  Furthermore, for any $x,y \in X$, where without loss of generality $x_{k(x,y)} = 1$ and $y_{k(x,y)} =0$, it holds by construction that $M(x)$ is contained in the half-space $V(\{k(x, y)\}, \emptyset)$ and $M(y)$ is contained in the complementary half-space $V(\emptyset, \{k(x,y)\})$, and so $M(x)$ and $M(y)$ are disjoint.
\end{proof}

This argument can be easily generalised to give similar projection lemmas for other classes of $H(n,k_1,k_2)$.

\end{document}